\documentclass[11pt]{article}
\usepackage{amsmath,latexsym,amsbsy,amssymb,enumerate,amsthm}
\usepackage{enumitem}
\usepackage{xcolor}
\usepackage{hyperref}
\usepackage[capitalise]{cleveref}
\usepackage{graphicx}
\usepackage{cases}
\usepackage{relsize}
\usepackage{subfigure}
\usepackage{float}
\usepackage{comment}
\usepackage[numbers]{natbib}
\usepackage{tikz}
\usepackage{forest}
\usepackage{tikz-qtree}
\usepackage{amsbsy}
\usetikzlibrary{tqft}
\usetikzlibrary{positioning}
\usetikzlibrary{arrows.meta, decorations.pathmorphing,calc,cd}
\def\epsilon{\varepsilon}

\newcommand{\set}[1]{\left\{#1\right\}}
\newcommand{\sett}[2]{\left\{#1 ~\Big| ~ #2\right\}}

\newcommand{\tcp}[2]{(#1+#2)^{tc}}
\newcommand{\tcu}[2]{(#1\cup #2)^{tc}}
\newcommand{\inv}[1]{\text{inv}\left(#1\right)}
\newcommand{\precdot}{\prec\mathrel{\mkern-5mu}\mathrel{\cdot}}
\newcommand{\rot}[3]{#1\overset{#3}\longrightarrow #2}
\newcommand{\card}[3]{\#_{#1}(#2,#3)}
\newcommand{\swo}{$s$-weak order}
\newcommand{\sdt}{$s$-decreasing tree}
\newcommand{\str}{$s$-tree rotation}

\newcommand{\trot}[3]{#1\overset{Tam#3}\longrightarrow #2}
\newcommand{\stt}{$s$-Tamari tree}
\newcommand{\staml}{$s$-Tamari lattice}

\newtheorem{theorem}{Theorem}[section]
\newtheorem{lemma}[theorem]{Lemma}
\newtheorem{corollary}[theorem]{Corollary}
\newtheorem{proposition}[theorem]{Proposition}
\newtheorem{definition}[theorem]{Definition}
\theoremstyle{remark}
\newtheorem{remark}[theorem]{Remark}

\newcounter{example}
\newenvironment{example}[1][]{\refstepcounter{example} \textbf{Example~\theexample.#1}}{}

\def\zh{\hat{0}}
\def\oneh{\hat{1}}

\title {Poset topology of $s$-weak order via SB-labelings}
\author{Stephen Lacina\footnote{The author was supported by NSF grants DMS-1953931 and DMS-1500987}}
\date{}

\begin{document}

\maketitle
\begin{abstract}
    Ceballos and Pons generalized weak order on permutations to a partial order on certain labeled trees, thereby introducing a new class of lattices called $s$-weak order. They also generalized the Tamari lattice by defining a particular sublattice of $s$-weak order called the $s$-Tamari lattice. We prove that the homotopy type of each open interval in $s$-weak order and in the $s$-Tamari lattice is either a ball or sphere. We do this by giving $s$-weak order and the $s$-Tamari lattice a type of edge labeling known as an SB-labeling. We characterize which intervals are homotopy equivalent to spheres and which are homotopy equivalent to balls; we also determine the dimension of the spheres for the intervals yielding spheres.
\end{abstract}

\begin{section}{Introduction}
In \cite{swkordceballospons2019}, Ceballos and Pons introduced a partial order called \textbf{{\swo}} on certain labeled trees known as \textbf{{\sdt}s}. They observed that this partial order generalizes weak order on permutations. They proved {\swo} is a lattice. They also found a particular class of {\sdt}s which play the role of 231-avoiding permutations. This led them to introduce a sublattice of {\swo} called the {\staml}, generalizing the Tamari lattice. 

Our main result is the following theorem:

\begin{theorem}\label{thm:introlabelthm}
The lattices {\swo} and the {\staml} each admit an SB-labeling. Thus, the order complex of each open interval in {\swo} and the {\staml} is homotopy equivalent to a ball or sphere of some dimension.
\end{theorem}

We prove this as \cref{thm:sblabeling} for {\swo} and \cref{thm:tamsblabelthm} for the {\staml}. In both cases, we prove topological results using the tool of SB-labelings developed by Hersh and M{\'e}sz{\'a}ros in \cite{sblabelhershmeszaros2017}. Our result generalizes another result of Hersh and M{\'e}sz{\'a}ros that weak order on permutations and the classical Tamari lattice admit SB-labelings, with our labelings specializing in those cases to SB-labelings distinct from theirs. 

In {\swo} and the {\staml}, the spheres in \cref{thm:introlabelthm} are not always top dimensional, demonstrating that these posets are not always shellable. See \cite{shellblposets} for example for the definition of a shellable poset. We intrinsically characterize which intervals in {\swo} and the {\staml} are homotopy equivalent to spheres and which are homotopy equivalent to balls. We also determine the dimension of the spheres for the intervals yielding homotopy spheres. As a corollary, we deduce that the M{\"o}bius functions of {\swo} and the {\staml} only take values in $\set{-1,0,1}$. It is also known that the existence of an SB-labeling implies that distinct sets of atoms in an interval have distinct joins, giving another consequence of our results.

Part of Ceballos and Pons' interest in {\swo} comes from geometry. They conjecture that the Hasse diagrams of {\swo} are the $1$-skeleta of polytopal subdivisions of polytopes. They call these potential polytopal complexes \textbf{$s$-permutahedra}. They also conjecture that in particular cases the polytopes they are subdividing are classical permutahedra. Our result of an SB-labeling for {\swo}, though it considers these lattices from a topological perspective, seems to provide two pieces of evidence for Ceballos and Pons' conjecture. The first piece of evidence is that the Hasse diagrams of many lattices which admit SB-labelings can be realized as the $1$-skeleta of polytopes. The second comes from the fact that Ceballos and Pons' geometric perspective is somewhat similar in flavor to one point of view in Hersh's work in \cite{oneskelhasseshersh2018}. Hersh studied posets which arise as the $1$-skeleta of simple polytopes via directing edges by some cost vector. In particular, Hersh's Theorem 4.9 in \cite{oneskelhasseshersh2018} proves that all open intervals in lattices which are realizable as such $1$-skeleta of simple polytopes are either homotopy balls or spheres.

Similarly, Ceballos and Pons' also took a geometric viewpoint on the {\staml}. They showed that the {\staml} is isomorphic to another generalization of the classical Tamari lattice, namely the $\nu$-Tamari lattice introduced by Pr{\'e}ville-Ratelle and Viennot in \cite{nutamlprvvien2017}. The geometry of the $\nu$-Tamari lattice was recently studied by Ceballos, Padrol, and Sarmiento in \cite{geomnutamlcebpadcam2019}. Similarly to how the Hasse diagram of the Tamari lattice is the $1$-skeleton of the associahedron, the Hasse diagram of the $\nu$-Tamari lattice is the $1$-skeleta of a polytopal subdivision of a polytope. Thus, the {\staml} also has such a realization. In the context of the {\staml}, Ceballos and Pons call these polytopal complexes \textbf{$s$-associahedra}. Further, they conjecture that in particular cases $s$-associahedra can be obtained from the $s$-permutahedra by deleting certain facets. The fact that the {\staml} admits an SB-labeling and has a realization as  the 1-skeleton of a polytopal complex seems to strengthen the evidence given by our result for Ceballos and Pons' conjecture of such realizations for $s$-permutahedra. Additionally, our result contributes two new classes of lattices which admit SB-labelings.

This paper proceeds as follows: \cref{sec:sword} provides the necessary background on posets, {\sdt}s, {\swo}, and the {\staml}. We largely follow the notation and definitions of \cite{swkordceballospons2019}. We also observe that {\swo} is not always a Cambrian lattice. \cref{sec:sword} reviews the notion of SB-labeling as well. \cref{sec:proofs} and \cref{sec:stam} are where we prove our main results, most notably giving SB-labelings for {\swo} and the {\staml}. 

\end{section}

\begin{section}{Background} \label{sec:sword}

\begin{subsection}{Background on Posets}
Let $(P,\leq)$ be a poset. For $x\leq y\in P$, the \textbf{closed interval} from $x$ to $y$ is the set $[x,y]=\sett{z\in P}{x\leq z\leq y}$. \textbf{The open interval from $x$ to $y$} is defined analogously with strict inequalities and denoted $(x,y)$. We say that $y$ \textbf{covers} $x$, denoted $x\lessdot y$, if $x\leq z\leq y$ implies $z=x$ or $z=y$. $P$ is a \textbf{lattice} if each pair $x,y\in P$ has a unique least upper bound, denoted $x\vee y$, and a unique greatest lower bound, denoted $x\wedge y$. We denote by $\zh$ (respectively $\oneh$) the unique minimal (respectively unique maximal) element of a finite lattice. The elements which cover $\zh$ are called \textbf{atoms}. For $x,y\in P$ with $x<y$, a \textbf{$k$-chain from $x$ to $y$} in $P$ is a subset $C=\set{x_0,x_1,\dots,x_k}\subset P$ such that $x=x_0<x_1<\dots< x_k=y$. A chain $C$ is said to be \textbf{saturated} if $x_i\lessdot x_{i+1}$ for all $i$. The \textbf{order complex} of $P$, denoted $\Delta(P)$, is the abstract simplicial complex with vertices the elements of $P$ and $i$-dimensional faces the $i$-chains of $P$. For $x,y\in P$ with $x<y$, we denote by $\Delta(x,y)$ the order complex of the open interval $(x,y)$ as an induced subposet of $P$. Thus, when we refer to topological properties of $P$, we mean the topological properties of a geometric realization of $\Delta(P)$. In particular, the homotopy type of $P$ refers to the homotopy type of $\Delta(P)$. It is well known that the M{\"o}bius function of $P$ $\mu_P$ satisfies $\mu_P(x,y) = \tilde{\chi}(\Delta(x,y))$. Here, $\tilde{\chi}$ is the reduced Euler characteristic. This provides one of the important connections between the combinatorial and enumerative structure of a poset and its topology.
\end{subsection}

\begin{subsection}{Background on {\swo}} \label{sec:bkgdswo}
A \textbf{weak composition} is a sequence of non-negative integers $s=(s(1),\dots, \newline s(n))$ with $s(i)\in \mathbb{N}$ for all $i\in [n]$. We say the \textbf{length} of a weak composition $s$ is $l(s)=n$. Let $s$ be a weak composition. An \textbf{$s$-decreasing tree} is a planar rooted tree $T$ with $n$ internal vertices which are labeled $1$ to $n$ (leaves are not labeled and are the only unlabeled vertices) such that internal vertex $i$ has $s(i)+1$ children and all labeled descendants of $i$ have labels less than $i$. The $s(i)+1$ children of $i$ are indexed by $0$ to $s(i)$. We denote the full subtree of $T$ rooted at $i$ by $T^i$, and denote the full subtrees rooted at the $s(i)+1$ children of $i$ by $T^i_0,\dots, T^i_{s(i)}$, respectively. For $i$ and $0\leq j\leq s(i)$, we denote by $T^i\setminus j$, the subtree of $T$ obtained from $T^i$ by replacing $T^i_j$ with a leaf. Also, $T^i_{j_1,\dots,j_k}$ will denote the forest of the full subtrees rooted at the $j_1,\dots, j_k$ children of $i$. Let $k$ be the $j$th child of $i$ in $T$. We define the \textbf{$j$th left subtree of $i$ in $T$}, denoted $_LT^i_j$, to be the subtree of $T$ with root $i$ obtained by walking from $i$ to $k$ and then down the left most subtree possible until reaching a leaf. Similarly, we define the \textbf{$j$th right most subtree of $i$ in $T$}, denoted $_RT^i_j$, to be the subtree of $T$ with root $i$ obtained by walking from $i$ to $k$ and then down the right most subtree possible until reaching a leaf. \cref{fig:sdtexample} is an example of an {\sdt} with $s=(0,0,0,2,1,3)$, along with some examples of the subtrees just defined

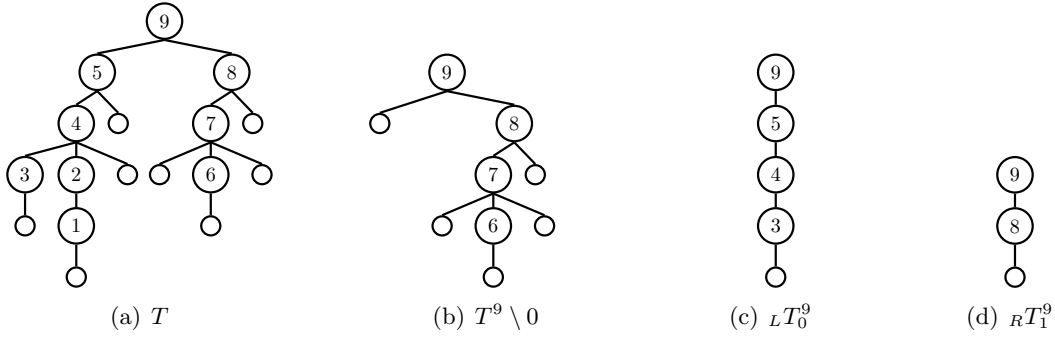
\begin{figure}[h]
    \centering
    \subfigure[$T$]{\scalebox{0.7}{\begin{tikzpicture}[very thick]
\node [style={draw,circle}]{9}
    child{ node[style={draw,circle},xshift=-15,yshift=15] {5} 
        child{ node[style={draw,circle},yshift=15,xshift=10] {4} 
            child{ node[style={draw,circle},yshift=15,xshift=15] {3} 
                child{ node[style={draw,circle},yshift=15,xshift=0] {}}
            }
            child{ node[style={draw,circle},yshift=15,xshift=0] {2}
                child{ node[style={draw,circle},yshift=15,xshift=0] {1}     child{ node[style={draw,circle},yshift=15,xshift=0] {}}
                }
            }
            child{ node[style={draw,circle},yshift=15,xshift=-15] {}}
        }
        child{ node[style={draw,circle},yshift=15,xshift=-10] {} }
    }
    child{ node[style={draw,circle},yshift=15,xshift=15] {8} 
        child{ node[style={draw,circle},yshift=15,xshift=10] {7}
            child{ node[style={draw,circle},yshift=15,xshift=15] {}}
            child{ node[style={draw,circle},yshift=15,xshift=0] {6} 
                child{ node[style={draw,circle},yshift=15,xshift=0] {}}
            }
            child{ node[style={draw,circle},yshift=15,xshift=-15] {}}
        }
        child{ node[style={draw,circle},yshift=15,xshift=-10] {}} 
    };
\end{tikzpicture}}}
\quad\quad\quad
    \subfigure[$T^9\setminus 0$]{\scalebox{0.7}{\begin{tikzpicture}[very thick]
\node [style={draw,circle}]{9}
    child{ node[style={draw,circle},xshift=-15,yshift=15] {}}
    child{ node[style={draw,circle},yshift=15,xshift=15] {8} 
        child{ node[style={draw,circle},yshift=15,xshift=10] {7}
            child{ node[style={draw,circle},yshift=15,xshift=15] {}}
            child{ node[style={draw,circle},yshift=15,xshift=0] {6} 
                child{ node[style={draw,circle},yshift=15,xshift=0] {}}
            }
            child{ node[style={draw,circle},yshift=15,xshift=-15] {}}
        }
        child{ node[style={draw,circle},yshift=15,xshift=-10] {}} 
    };
\end{tikzpicture} \quad\quad}}
\quad\quad\quad
\subfigure[$_LT^9_0$]{\quad\quad \scalebox{0.7}{\begin{tikzpicture}[very thick]
\node [style={draw,circle}]{9}
    child{ node[style={draw,circle},xshift=0,yshift=15] {5}
        child{node[style={draw,circle},yshift=15,xshift=0] {4}
            child{node[style={draw,circle},yshift=15,xshift=0] {3} 
                child{node[style={draw,circle},yshift=15,xshift=0] {} }
            }
        }
    };
\end{tikzpicture} \quad\quad}}
\quad\quad\quad
\subfigure[$_RT^9_1$]{\quad\quad \scalebox{0.7}{\begin{tikzpicture}[very thick]
\node [style={draw,circle}]{9}
    child{ node[style={draw,circle},xshift=0,yshift=15] {8}
        child{ node[style={draw,circle},yshift=15,xshift=0] {} }
    };
\end{tikzpicture} \quad\quad}}
    \caption{An {\sdt} $T$ with $s=(0,0,0,2,1,0,2,1,1)$ and examples of some defined subtrees.}
    \label{fig:sdtexample}
\end{figure}

\begin{definition}\label{def:treeinversions}\cite[Definition 2.1]{swkordceballospons2019}
Let $T$ be an $s$-decreasing tree and $1\leq x < y \leq n$. The \textbf{cardinality} of $(y,x)$ in $T$, denoted $\boldsymbol{\card{T}{y}{x}}$, is defined by the following rules: \begin{enumerate}
\item $\card{T}{y}{x}=0$ if $x$ is left of $y$ in $T$ or $x\in T^y_0$; 
\item $\card{T}{y}{x}=i$ if $x\in T^y_i$ with $0<i<s(y)$; and
\item $\card{T}{y}{x}=s(y)$ if $x\in T^y_{s(y)}$ or $x$ is right of $y$ in $T$. \end{enumerate}

If $\card{T}{y}{x}>0$, then $(y,x)$ is said to be a \textbf{tree inversion} of $T$. We denote by $\textbf{inv}\boldsymbol{(T)}$ the multi-set of tree inversions of $T$ counted with multiplicity their cardinality.
\end{definition}

Now we can also formally describe the $j$th left and right subtrees of $i$ in $T$, examples of which are found in (c) and (d) of \cref{fig:sdtexample}. \[ _LT^i_j=\sett{d\in T^i}{d=i, \text{ or } d\in T^i_j \text{ and } \card{T}{e}{d}=0 \text{ } \forall e\in T^i_j \text{ such that } d<e }.\] \[ _RT^i_j=\sett{d\in T^i}{d=i, \text{ or } d\in T^i_j \text{ and } \card{T}{e}{d}=s(e) \text{ } \forall e\in T^i_j \text{ such that } d<e }. \]

\begin{remark}
For $s=(1,\dots,1)$, $s$-decreasing trees are in by bijection with permutations in $S_{l(s)}$ and tree inversions are precisely inversions of the corresponding permutation. 
\end{remark}

\begin{remark}\label{rmk:guarantreecontain}
If $T$ is an \sdt, $1\leq a<b\leq n$, and $0<\card{T}{b}{a}<s(b)$, then $a\in T^b_{\card{T}{b}{a}}$.
\end{remark}

\begin{remark}\label{rmk:biggercontain}
If $e\in T^a$ and $e\in T^b_i$ for some $a<b$, then $a\in T^b_i$. Further, if $e\in T^a$ and $a<b$, then $\card{T}{b}{e}=\card{T}{b}{a}$.
\end{remark}

\cref{fig:treeandinvs} is an {\sdt} with the cardinality of each pair of labeled vertices listed.

\begin{figure}[H] 
\centering
\subfigure{
\centering
\scalebox{0.8}{\begin{tikzpicture}[very thick]
\node [style={draw,circle}]{6}
    child{ node[style={draw,circle},xshift=10,yshift=15] {5} 
        child{ node[style={draw,circle},yshift=15,xshift=10] {} }
        child{ node[style={draw,circle},yshift=15,xshift=-10] {} }
    }
    child{ node[style={draw,circle},yshift=15] {} }
    child{ node[style={draw,circle},xshift=-15,yshift=15] {4} 
        child{node[style={draw,circle},yshift=15,xshift=15] {3}
            child{node[style={draw,circle},yshift=15,xshift=0] {} }
        }
        child{node[style={draw,circle},yshift=15] {} }
        child{ node[style={draw,circle},yshift=15,xshift=-15] {2}
            child{ node[style={draw,circle},yshift=15] {1}
                child{node[style={draw,circle},yshift=15] {} }
            }
        }
    }
    child{ node[style={draw,circle},yshift=15,xshift=-15] {} };
\end{tikzpicture}
}}
\subfigure{
\centering
\scalebox{0.9}{\begin{tabular}{ c c c c c }
 $\card{T}{6}{5}=0$ & $\card{T}{6}{4}=2$ & $\card{T}{6}{3}=2$ & $\card{T}{6}{2}=2$ & $\card{T}{6}{1}=2$ \\
 & $\card{T}{5}{4}=1$ & $\card{T}{5}{3}=1$ & $\card{T}{5}{2}=1$ & $\card{T}{5}{1}=1$ \\ 
  & & $\card{T}{4}{3}=0$ & $\card{T}{4}{2}=2$ & $\card{T}{4}{1}=2$ \\  
  & & & $\card{T}{3}{2}=0$ & $\card{T}{3}{1}=0$ \\
  & & & & $\card{T}{2}{1}=0$ \\
  & & & & \\
  & & & & \\
  & & & & \\
  & & & & \\
  & & & & \\
  & & & & \\
  & & & & \\
  & & & & \\
  & & & & \\
  & & & & \\
\end{tabular}
}}
\vspace{-30mm}
\caption{An {\sdt} and its cardinalities for $s=(0,0,0,2,1,3)$.} \label{fig:treeandinvs}
\end{figure}
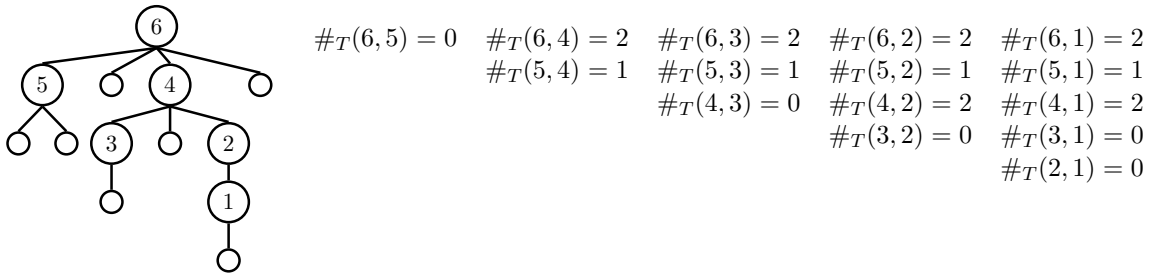

Next we establish notation for sets of tree inversions examples of which follow \cref{fig:swoexamples} using {\sdt}s from those examples of {\swo}.

\begin{definition}\label{def:multinvsets}\cite[Definition 2.2]{swkordceballospons2019}
A \textbf{multi-inversion set} on $[n]$ is a multi-set $I$ of pairs $(y,x)$ such that $1\leq x < y \leq n$. We write $\boldsymbol{\#_{I}(y,x)}$ for the multiplicity of $(y,x)$ in $I$ so if $(y,x)$ does not appear in $I$, $\card{I}{y}{x}=0$. 

Given multi-inversion sets $I$ and $J$, we say $I$ is \textbf{included} in $J$ and write $I\subseteq J$ if $\card{I}{y}{x}\leq \card{J}{y}{x}$ for all $1\leq x<y\leq n$. We also define the \textbf{multi-inversion set difference}, $\boldsymbol{J-I}$, to be the multi-inversion set with $\card{J-I}{y}{x} = \card{J}{y}{x}-\card{I}{y}{x}$ whenever this difference is non-negative and 0 otherwise.
\end{definition}

 This leads to a characterization of those multi-inversion sets which are actually sets of tree inversions of {\sdt}s. Further, it motivates the definition of {\swo} in analogy with the inversion set definition definition of weak order on permutations.

\begin{proposition}\label{prop:invchartrees}\cite[Proposition 2.4]{swkordceballospons2019}
There is a bijection between $s$-decreasing trees and multi-inversion sets $I$ satisfying $\card{I}{y}{x}\leq s(y)$ and the following two properties:
\begin{itemize}
    \item \textbf{Transitivity}: if $a<b<c$ and $\card{I}{c}{b}=i$, then $\card{I}{b}{a}=0$ or $\card{I}{c}{a}\geq i$.
    \item \textbf{Planarity}: if $a<b<c$ and $\card{I}{c}{a}=i$, then $\card{I}{b}{a}=s(b)$ or $\card{I}{c}{b}\geq i$.
\end{itemize}
Such multi-inversion sets are called \textbf{$s$-tree inversion sets}.
\end{proposition}

\begin{definition}\label{def:swkorder}\cite[Definition 2.5]{swkordceballospons2019}
Let $s$ be a weak composition. The \textbf{$s$-weak order} is the partial order on {\sdt}s given by $T\preceq Z$ if and only if $\inv{Z} \subseteq \inv{T}$ for {\sdt}s $T$ and $Z$ using the inclusion of multi-inversion sets from \cref{def:multinvsets}.
\end{definition}

\cref{fig:swoexamples} shows three examples of \swo. The labelings of the last two examples is our SB-labeling which is defined in \cref{sec:proofs}.

\newsavebox{\minbox}
\sbox{\minbox}{
   \scalebox{.4}{\begin{tikzpicture}[very thick]
\node [style={draw,circle}]{3}
    child{ node[style={draw,circle},xshift=15,yshift=15] {2} 
        child{ node[style={draw,circle},yshift=15] {1} 
            child{ node[style={draw,circle},yshift=15] {}}
            }
        }
    child{ node[style={draw,circle},yshift=15] {} }
    child{ node[style={draw,circle},xshift=-15,yshift=15] {} };
\end{tikzpicture}}
 }

\newsavebox{\twothreebox}
\sbox{\twothreebox}{
   \scalebox{.4}{\begin{tikzpicture}[very thick]
\node [style={draw,circle}]{3}
        child{ node[style={draw,circle},xshift=15,yshift=15] {1} 
            child{ node[style={draw,circle},yshift=15] {}}
        }
    child{ node[style={draw,circle},yshift=15] {2} 
        child{ node[style={draw,circle},yshift=15] {} }
        }
    child{ node[style={draw,circle},xshift=-15,yshift=15] {} };
\end{tikzpicture}}
 } 
 
\newsavebox{\onethreebox}
\sbox{\onethreebox}{
   \scalebox{.4}{\begin{tikzpicture}[very thick]
\node [style={draw,circle}]{3}
        child{ node[style={draw,circle},xshift=15,yshift=15] {2} 
            child{ node[style={draw,circle},yshift=15] {}}
        }
    child{ node[style={draw,circle},yshift=15] {1} 
        child{ node[style={draw,circle},yshift=15] {} }
        }
    child{ node[style={draw,circle},xshift=-15,yshift=15] {} };
\end{tikzpicture}}
 } 

\newsavebox{\onetwothreebox}
\sbox{\onetwothreebox}{
   \scalebox{.4}{\begin{tikzpicture}[very thick]
\node [style={draw,circle}]{3}
        child{ node[style={draw,circle},xshift=15,yshift=15] {}}
    child{ node[style={draw,circle},yshift=15] {2} 
        child{ node[style={draw,circle},yshift=15] {1} 
            child{ node[style={draw,circle},yshift=15] {} }
            }
        }
    child{ node[style={draw,circle},xshift=-15,yshift=15] {} };
\end{tikzpicture}}
 } 
 
\newsavebox{\twotwothreebox}
\sbox{\twotwothreebox}{
   \scalebox{.4}{\begin{tikzpicture}[very thick]
\node [style={draw,circle}]{3}
        child{ node[style={draw,circle},xshift=15,yshift=15] {1} 
            child{ node[style={draw,circle},yshift=15] {}}
        }
    child{ node[style={draw,circle}] {} }
    child{ node[style={draw,circle},xshift=-15,yshift=15] {2} 
        child{ node[style={draw,circle},yshift=15] {} }
        };
\end{tikzpicture}}
 }
 
\newsavebox{\oneonethreebox}
\sbox{\oneonethreebox}{
   \scalebox{.4}{\begin{tikzpicture}[very thick]
\node [style={draw,circle}]{3}
        child{ node[style={draw,circle},xshift=15,yshift=15] {2} 
            child{ node[style={draw,circle},yshift=15] {}}
        }
    child{ node[style={draw,circle},yshift=15] {} }
    child{ node[style={draw,circle},xshift=-15,yshift=15] {1} 
        child{ node[style={draw,circle},yshift=15] {} }
        };
\end{tikzpicture}}
 }
 
\newsavebox{\oneonetwobox}
\sbox{\oneonetwobox}{
   \scalebox{.4}{\begin{tikzpicture}[very thick]
\node [style={draw,circle}]{3}
    child{ node[style={draw,circle},xshift=15,yshift=15] {} }
    child{ node[style={draw,circle},yshift=15] {2} 
        child{ node[style={draw,circle},yshift=15] {}}
        }
    child{ node[style={draw,circle},xshift=-15,yshift=15] {1} 
        child{ node[style={draw,circle},yshift=15] {} }
        };
\end{tikzpicture}}
 }
 
\newsavebox{\twotwoonebox}
\sbox{\twotwoonebox}{
   \scalebox{.4}{\begin{tikzpicture}[very thick]
\node [style={draw,circle}]{3}
    child{ node[style={draw,circle},xshift=15,yshift=15] {} }
    child{ node[style={draw,circle},yshift=15] {1} 
        child{ node[style={draw,circle},yshift=15] {}}
        }
    child{ node[style={draw,circle},xshift=-15,yshift=15] {2} 
        child{ node[style={draw,circle},yshift=15] {} }
        };
\end{tikzpicture}}
 }
 
\newsavebox{\maxbox}
\sbox{\maxbox}{
   \scalebox{.4}{\begin{tikzpicture}[very thick]
\node [style={draw,circle}]{3}
    child{ node[style={draw,circle},xshift=15,yshift=15] {} }
    child{ node[style={draw,circle},yshift=15] {} }
    child{ node[style={draw,circle},xshift=-15,yshift=15] {2} 
        child{ node[style={draw,circle},yshift=15] {1}
            child{node[style={draw,circle},yshift=15] {} }
            }
        };
\end{tikzpicture}}
 }
 
\newsavebox{\zotminbox}
\sbox{\zotminbox}{
   \scalebox{.4}{\begin{tikzpicture}[very thick]
\node [style={draw,circle}]{3}
    child{ node[style={draw,circle},xshift=15,yshift=15] {2} 
        child{ node[style={draw,circle},yshift=15,xshift=0] {1} 
            child{ node[style={draw,circle},yshift=15,xshift=0] {}}
            }
        child{ node[style={draw,circle},yshift=15,xshift=0] {} }
        }
    child{ node[style={draw,circle},yshift=15] {} }
    child{ node[style={draw,circle},xshift=-15,yshift=15] {} };
\end{tikzpicture}}
 }
 
\newsavebox{\zotbox}
\sbox{\zotbox}{
   \scalebox{.4}{\begin{tikzpicture}[very thick]
\node [style={draw,circle}]{3}
    child{ node[style={draw,circle},xshift=15,yshift=15] {2}
        child{ node[style={draw,circle},yshift=15,xshift=0] {} }
        child{ node[style={draw,circle},yshift=15,xshift=0] {1} 
            child{ node[style={draw,circle},yshift=15,xshift=0] {}}
            }
        }
    child{ node[style={draw,circle},yshift=15] {} }
    child{ node[style={draw,circle},xshift=-15,yshift=15] {} };
\end{tikzpicture}}
 }
 
\newsavebox{\zotothbox}
\sbox{\zotothbox}{
   \scalebox{.4}{\begin{tikzpicture}[very thick]
\node [style={draw,circle}]{3}
    child{ node[style={draw,circle},xshift=15,yshift=15] {2}
        child{ node[style={draw,circle},yshift=15,xshift=10] {} }
        child{ node[style={draw,circle},yshift=15,xshift=-10] {} }
        }
    child{ node[style={draw,circle},yshift=15,xshift=0] {1} 
            child{ node[style={draw,circle},yshift=15,xshift=0] {}}
            }
    child{ node[style={draw,circle},xshift=-15,yshift=15] {} };
\end{tikzpicture}}
 }
 
\newsavebox{\zotothtthbox}
\sbox{\zotothtthbox}{
   \scalebox{.4}{\begin{tikzpicture}[very thick]
\node [style={draw,circle}]{3}
    child{ node[style={draw,circle},xshift=15,yshift=15] {} }
    child{ node[style={draw,circle},xshift=0,yshift=15] {2}
        child{ node[style={draw,circle},yshift=15,xshift=0] {}}
        child{ node[style={draw,circle},yshift=15,xshift=0] {1} 
            child{ node[style={draw,circle},yshift=15,xshift=0] {}}
            }
        }
    child{ node[style={draw,circle},xshift=-15,yshift=15] {} };
\end{tikzpicture}}
 }
 
\newsavebox{\zotothothbox}
\sbox{\zotothothbox}{
   \scalebox{.4}{\begin{tikzpicture}[very thick]
\node [style={draw,circle}]{3}
    child{ node[style={draw,circle},xshift=15,yshift=15] {2}
        child{ node[style={draw,circle},yshift=15,xshift=10] {} }
        child{ node[style={draw,circle},yshift=15,xshift=-10] {} }
        }
    child{ node[style={draw,circle},xshift=0,yshift=15] {} }
    child{ node[style={draw,circle},yshift=15,xshift=-15] {1} 
            child{ node[style={draw,circle},yshift=15,xshift=0] {}}
            };
\end{tikzpicture}}
 }
 
\newsavebox{\zotothothtthbox}
\sbox{\zotothothtthbox}{
   \scalebox{.4}{\begin{tikzpicture}[very thick]
\node [style={draw,circle}]{3}
    child{ node[style={draw,circle},xshift=15,yshift=15] {} }
    child{ node[style={draw,circle},xshift=0,yshift=15] {2}
        child{ node[style={draw,circle},yshift=15,xshift=10] {} }
        child{ node[style={draw,circle},yshift=15,xshift=-10] {} }
        }
    child{ node[style={draw,circle},yshift=15,xshift=-15] {1} 
            child{ node[style={draw,circle},yshift=15,xshift=0] {}}
            };
\end{tikzpicture}}
 }

\newsavebox{\zmaxbox}
\sbox{\zmaxbox}{
   \scalebox{.4}{\begin{tikzpicture}[very thick]
\node [style={draw,circle}]{3}
    child{ node[style={draw,circle},xshift=15,yshift=15] {} }
    child{ node[style={draw,circle},xshift=0,yshift=15] {} }
    child{ node[style={draw,circle},xshift=-15,yshift=15] {2}
        child{ node[style={draw,circle},yshift=15,xshift=0] {}}
        child{ node[style={draw,circle},yshift=15,xshift=0] {1} 
            child{ node[style={draw,circle},yshift=15,xshift=0] {}}
            }
        };
\end{tikzpicture}}
 }
 
\newsavebox{\ztthbox}
\sbox{\ztthbox}{
   \scalebox{.4}{\begin{tikzpicture}[very thick]
\node [style={draw,circle}]{3}
    child{ node[style={draw,circle},yshift=15,xshift=15] {1} 
            child{ node[style={draw,circle},yshift=15,xshift=0] {}}
            }
    child{ node[style={draw,circle},xshift=0,yshift=15] {2}
        child{ node[style={draw,circle},yshift=15,xshift=10] {} }
        child{ node[style={draw,circle},yshift=15,xshift=-10] {} }
        }
    child{ node[style={draw,circle},xshift=-15,yshift=15] {} };
\end{tikzpicture}}
 }

\newsavebox{\ztthtthbox}
\sbox{\ztthtthbox}{
   \scalebox{.4}{\begin{tikzpicture}[very thick]
\node [style={draw,circle}]{3}
    child{ node[style={draw,circle},yshift=15,xshift=15] {1} 
            child{ node[style={draw,circle},yshift=15,xshift=0] {}}
            }
    child{ node[style={draw,circle},xshift=0,yshift=15] {} }
    child{ node[style={draw,circle},xshift=-15,yshift=15] {2}
        child{ node[style={draw,circle},yshift=15,xshift=10] {} }
        child{ node[style={draw,circle},yshift=15,xshift=-10] {} }
        };
\end{tikzpicture}}
 }
 
\newsavebox{\ztthtthothbox}
\sbox{\ztthtthothbox}{
   \scalebox{.4}{\begin{tikzpicture}[very thick]
\node [style={draw,circle}]{3}
    child{ node[style={draw,circle},xshift=15,yshift=15] {} }
    child{ node[style={draw,circle},yshift=15,xshift=0] {1} 
            child{ node[style={draw,circle},yshift=15,xshift=0] {}}
            }
    child{ node[style={draw,circle},xshift=-15,yshift=15] {2}
        child{ node[style={draw,circle},yshift=15,xshift=10] {} }
        child{ node[style={draw,circle},yshift=15,xshift=-10] {} }
        };
\end{tikzpicture}}
 }
 
\newsavebox{\ztthothbox}
\sbox{\ztthothbox}{
   \scalebox{.4}{\begin{tikzpicture}[very thick]
\node [style={draw,circle}]{3}
    child{ node[style={draw,circle},yshift=15,xshift=15] {} }
    child{ node[style={draw,circle},xshift=0,yshift=15] {2} 
        child{ node[style={draw,circle},yshift=15,xshift=0] {1} 
            child{ node[style={draw,circle},yshift=15,xshift=0] {}}
            }
        child{ node[style={draw,circle},yshift=15,xshift=0] {} }
        }
    child{ node[style={draw,circle},xshift=-15,yshift=15] {} };
\end{tikzpicture}}
 }
 
\newsavebox{\ztthtthothothbox}
\sbox{\ztthtthothothbox}{
   \scalebox{.4}{\begin{tikzpicture}[very thick]
\node [style={draw,circle}]{3}
    child{ node[style={draw,circle},yshift=15,xshift=15] {} }
    child{ node[style={draw,circle},xshift=0,yshift=15] {} }
    child{ node[style={draw,circle},xshift=-15,yshift=15] {2} 
        child{ node[style={draw,circle},yshift=15,xshift=0] {1} 
            child{ node[style={draw,circle},yshift=15,xshift=0] {}}
            }
        child{ node[style={draw,circle},yshift=15,xshift=0] {} }
        };
\end{tikzpicture}}
 }

\newsavebox{\ttminbox}
\sbox{\ttminbox}{
   \scalebox{.4}{\begin{tikzpicture}[very thick]
\node [style={draw,circle}]{3}
    child{ node[style={draw,circle},xshift=15,yshift=15] {2} 
        child{ node[style={draw,circle},yshift=15,xshift=15] {1} 
            child{ node[style={draw,circle},yshift=15] {}}
            }
        child{ node[style={draw,circle},yshift=15] {} }
        child{ node[style={draw,circle},yshift=15,xshift=-15] {} }
        }
    child{ node[style={draw,circle},yshift=15] {} }
    child{ node[style={draw,circle},xshift=-15,yshift=15] {} };
\end{tikzpicture}}
 }
 
\newsavebox{\tthbox}
\sbox{\tthbox}{
   \scalebox{.4}{\begin{tikzpicture}[very thick]
\node [style={draw,circle}]{3}
    child{ node[style={draw,circle},yshift=15] {1} 
        child{ node[style={draw,circle},yshift=15] {}}
        }
    child{ node[style={draw,circle},yshift=15] {2} 
        child{ node[style={draw,circle},yshift=15,xshift=15] {}}
        child{ node[style={draw,circle},yshift=15] {} }
        child{ node[style={draw,circle},yshift=15,xshift=-15] {} }
        }
    child{ node[style={draw,circle},xshift=-15,yshift=15] {} };
\end{tikzpicture}}
 } 
 
\newsavebox{\tthothbox}
\sbox{\tthothbox}{
   \scalebox{.4}{\begin{tikzpicture}[very thick]
\node [style={draw,circle}]{3}
    child{ node[style={draw,circle},xshift=15,yshift=15] {} }
    child{ node[style={draw,circle},yshift=15] {2} 
        child{ node[style={draw,circle},yshift=15,xshift=15] {1} 
            child{ node[style={draw,circle},yshift=15] {}}
            }
        child{ node[style={draw,circle},yshift=15] {} }
        child{ node[style={draw,circle},yshift=15,xshift=-15] {} }
        }
    child{ node[style={draw,circle},xshift=-15,yshift=15] {} };
\end{tikzpicture}}
 }

\newsavebox{\otbox}
\sbox{\otbox}{
   \scalebox{.4}{\begin{tikzpicture}[very thick]
\node [style={draw,circle}]{3}
    child{ node[style={draw,circle},xshift=15,yshift=15] {2} 
        child{ node[style={draw,circle},yshift=15,xshift=15] {} }
        child{ node[style={draw,circle},yshift=15] {1} 
            child{ node[style={draw,circle},yshift=15] {}}
            }
        child{ node[style={draw,circle},yshift=15,xshift=-15] {} }
        }
    child{ node[style={draw,circle},yshift=15] {} }
    child{ node[style={draw,circle},xshift=-15,yshift=15] {} };
\end{tikzpicture}}
 } 

\newsavebox{\ottthbox}
\sbox{\ottthbox}{
   \scalebox{.4}{\begin{tikzpicture}[very thick]
\node [style={draw,circle}]{3}
    child{ node[style={draw,circle},yshift=15,xshift=15] {} }
    child{ node[style={draw,circle},yshift=15] {2} 
        child{ node[style={draw,circle},yshift=15,xshift=15] {} }
        child{ node[style={draw,circle},yshift=15] {1} 
            child{ node[style={draw,circle},yshift=15] {}}
            }
        child{ node[style={draw,circle},yshift=15,xshift=-15] {} }
        }
    child{ node[style={draw,circle},xshift=-15,yshift=15] {} };
\end{tikzpicture}}
 }

\newsavebox{\ottthotbox}
\sbox{\ottthotbox}{
   \scalebox{.4}{\begin{tikzpicture}[very thick]
\node [style={draw,circle}]{3}
    child{ node[style={draw,circle},yshift=15,xshift=15] {} }
    child{ node[style={draw,circle},yshift=15] {2} 
        child{ node[style={draw,circle},yshift=15,xshift=15] {} }
        child{ node[style={draw,circle},yshift=15,xshift=0] {} }
        child{ node[style={draw,circle},yshift=15,xshift=-15] {1} 
            child{ node[style={draw,circle},yshift=15] {}}
            }
        }
    child{ node[style={draw,circle},xshift=-15,yshift=15] {} };
\end{tikzpicture}}
 }
 
\newsavebox{\tthtthbox}
\sbox{\tthtthbox}{
   \scalebox{.4}{\begin{tikzpicture}[very thick]
\node [style={draw,circle}]{3}
    child{ node[style={draw,circle},yshift=15,xshift=15] {1} 
        child{ node[style={draw,circle},yshift=15] {}}
        }
    child{ node[style={draw,circle},xshift=0,yshift=15] {} }
    child{ node[style={draw,circle},yshift=15,xshift=-15] {2} 
        child{ node[style={draw,circle},yshift=15,xshift=15] {}}
        child{ node[style={draw,circle},yshift=15] {} }
        child{ node[style={draw,circle},yshift=15,xshift=-15] {} }
        };
\end{tikzpicture}}
 }
 
\newsavebox{\ototbox}
\sbox{\ototbox}{
   \scalebox{.4}{\begin{tikzpicture}[very thick]
\node [style={draw,circle}]{3}
    child{ node[style={draw,circle},xshift=15,yshift=15] {2} 
        child{ node[style={draw,circle},yshift=15,xshift=15] {} }
        child{ node[style={draw,circle},yshift=15] {} }
        child{ node[style={draw,circle},yshift=15,xshift=-15] {1} 
            child{ node[style={draw,circle},yshift=15] {}}
            }
        }
    child{ node[style={draw,circle},yshift=15] {} }
    child{ node[style={draw,circle},xshift=-15,yshift=15] {} };
\end{tikzpicture}}
 }
 
\newsavebox{\ototothbox}
\sbox{\ototothbox}{
   \scalebox{.4}{\begin{tikzpicture}[very thick]
\node [style={draw,circle}]{3}
    child{ node[style={draw,circle},xshift=0,yshift=15] {2} 
        child{ node[style={draw,circle},yshift=15,xshift=15] {} }
        child{ node[style={draw,circle},yshift=15] {} }
        child{ node[style={draw,circle},yshift=15,xshift=-15] {} }
        }
    child{ node[style={draw,circle},yshift=15] {1} 
            child{ node[style={draw,circle},yshift=15] {}}
            }
    child{ node[style={draw,circle},xshift=-15,yshift=15] {} };
\end{tikzpicture}}
 }
 
\newsavebox{\ototothothbox}
\sbox{\ototothothbox}{
   \scalebox{.4}{\begin{tikzpicture}[very thick]
\node [style={draw,circle}]{3}
    child{ node[style={draw,circle},xshift=15,yshift=15] {2} 
        child{ node[style={draw,circle},yshift=15,xshift=15] {} }
        child{ node[style={draw,circle},yshift=15] {} }
        child{ node[style={draw,circle},yshift=15,xshift=-15] {} }
        }
    child{ node[style={draw,circle},xshift=0,yshift=15] {} }
    child{ node[style={draw,circle},yshift=15,xshift=-15] {1} 
            child{ node[style={draw,circle},yshift=15] {}}
            };
\end{tikzpicture}}
 }
 
\newsavebox{\otothotothbox}
\sbox{\otothotothbox}{
   \scalebox{.4}{\begin{tikzpicture}[very thick]
\node [style={draw,circle}]{3}
    child{ node[style={draw,circle},xshift=15,yshift=15] {} }
    child{ node[style={draw,circle},xshift=0,yshift=15] {2} 
        child{ node[style={draw,circle},yshift=15,xshift=15] {} }
        child{ node[style={draw,circle},yshift=15] {} }
        child{ node[style={draw,circle},yshift=15,xshift=-15] {} }
        }
    child{ node[style={draw,circle},yshift=15,xshift=0] {1} 
            child{ node[style={draw,circle},yshift=15] {}}
            };
\end{tikzpicture}}
 }
 
\newsavebox{\tthtthothbox}
\sbox{\tthtthothbox}{
   \scalebox{.4}{\begin{tikzpicture}[very thick]
\node [style={draw,circle}]{3}
    child{ node[style={draw,circle},xshift=15,yshift=15] {} }
    child{ node[style={draw,circle},yshift=15,xshift=0] {1} 
        child{ node[style={draw,circle},yshift=15] {}}
        }
    child{ node[style={draw,circle},yshift=15,xshift=0] {2} 
        child{ node[style={draw,circle},yshift=15,xshift=15] {}}
        child{ node[style={draw,circle},yshift=15] {} }
        child{ node[style={draw,circle},yshift=15,xshift=-15] {} }
        };
\end{tikzpicture}}
 }
 
\newsavebox{\tthtthothothbox}
\sbox{\tthtthothothbox}{
   \scalebox{.4}{\begin{tikzpicture}[very thick]
\node [style={draw,circle}]{3}
    child{ node[style={draw,circle},xshift=15,yshift=15] {} }
    child{ node[style={draw,circle},xshift=0,yshift=15] {} }
    child{ node[style={draw,circle},yshift=15,xshift=-15] {2} 
        child{ node[style={draw,circle},yshift=15,xshift=15] {1} 
            child{ node[style={draw,circle},yshift=15] {}}
            }
        child{ node[style={draw,circle},yshift=15] {} }
        child{ node[style={draw,circle},yshift=15,xshift=-15] {} }
        };
\end{tikzpicture}}
 }

\newsavebox{\tthtthothothotbox}
\sbox{\tthtthothothotbox}{
   \scalebox{.4}{\begin{tikzpicture}[very thick]
\node [style={draw,circle}]{3}
    child{ node[style={draw,circle},xshift=15,yshift=15] {} }
    child{ node[style={draw,circle},xshift=0,yshift=15] {} }
    child{ node[style={draw,circle},yshift=15,xshift=-15] {2} 
        child{ node[style={draw,circle},yshift=15,xshift=15] {} }
        child{ node[style={draw,circle},yshift=15,xshift=0] {1} 
            child{ node[style={draw,circle},yshift=15] {}}
            }
        child{ node[style={draw,circle},yshift=15,xshift=-15] {} }
        };
\end{tikzpicture}}
 }
 
\newsavebox{\ttmaxbox}
\sbox{\ttmaxbox}{
   \scalebox{.4}{\begin{tikzpicture}[very thick]
\node [style={draw,circle}]{3}
    child{ node[style={draw,circle},xshift=15,yshift=15] {} }
    child{ node[style={draw,circle},xshift=0,yshift=15] {} }
    child{ node[style={draw,circle},yshift=15,xshift=-15] {2} 
        child{ node[style={draw,circle},yshift=15,xshift=15] {} }
        child{ node[style={draw,circle},yshift=15,xshift=0] {} }
        child{ node[style={draw,circle},yshift=15,xshift=-15] {1} 
            child{ node[style={draw,circle},yshift=15] {}}
            }
        };
\end{tikzpicture}}
 }
 
\begin{figure}[H]
    \centering
    \subfigure[$s=(0,0,2)$]{\scalebox{0.6}{\begin{tikzpicture}[very thick]
  \node (max) at (0,8) {\usebox{\maxbox}};  
  \node (h) at (-1,5.5) {\usebox{\twotwoonebox}};
  \node (g) at (1,5.5) {\usebox{\oneonetwobox}};
  \node (a) at (-2,3) {\usebox{\twotwothreebox}};
  \node (b) at (0,3) {\usebox{\onetwothreebox}};
  \node (c) at (2,3) {\usebox{\oneonethreebox}};
  \node (e) at (-1,0.5) {\usebox{\twothreebox}};
  \node (f) at (1,0.5) {\usebox{\onethreebox}};
  \node (min) at (0,-2) {\usebox{\minbox}};
  \draw (min) -- (f) -- (c) -- (g) -- (max) -- (h) -- (b) -- (g);
  \draw (min) -- (e) -- (a) -- (h);
  \draw (e) -- (b) -- (f);
\end{tikzpicture}}}
    \subfigure[$s=(0,1,2)$]{\scalebox{.6}{\begin{tikzpicture}[very thick]
  \node (12131323) at (0,6) {\usebox{\zotothothtthbox}}; 
  \node (max) at (-2,8) {\usebox{\zmaxbox}};
  \node (23231313) at (-4,6) {\usebox{\ztthtthothothbox}};
  \node (232313) at (-4,4) {\usebox{\ztthtthothbox}};
  \node (2323) at (-4,2) {\usebox{\ztthtthbox}};
  \node (121323) at (0,4) {\usebox{\zotothtthbox}};
  \node (1213) at (2,2) {\usebox{\zotothbox}};
  \node (121313) at (2,4) {\usebox{\zotothothbox}};
  \node (2313) at (-2,2) {\usebox{\ztthothbox}};
  \node (23) at (-2,0) {\usebox{\ztthbox}};
  \node (12) at (2,0) {\usebox{\zotbox}};
  \node (min) at (0,-2) {\usebox{\zotminbox}};
  \draw (min) --node [below left,blue]{$2$} (23);
  \draw (23) --node [left,blue]{$1$} (2313);
  \draw (min) --node [below right,blue]{$1$} (12);
  \draw (23) --node [below left,blue]{$2$} (2323);
  \draw (2323) --node [left,blue]{$1$} (232313);
  \draw (232313) --node [left,blue]{$1$} (23231313);
  \draw (23231313) --node [above left,blue]{$1$} (max);
  \draw (12) --node [left,blue]{$1$} (1213);
  \draw (1213) --node [left,blue]{$1$} (121313);
  \draw (1213) --node [below left,blue]{$2$} (121323);
  \draw (121323) --node [left,blue]{$1$} (12131323);
  \draw (121313) --node [below left,blue]{$2$} (12131323);
  \draw (12131323) --node [below left,blue]{$2$} (max);
  \draw (232313) --node [left,blue]{$1$} (23231313);
  \draw (2313) --node [above left,blue]{$1$} (121323);
  \draw (2313) --node [below left,blue]{$2$} (232313);
\end{tikzpicture}}}
\subfigure[$s=(0,2,2)$]{\scalebox{.6}{\begin{tikzpicture}[very thick]
  \node (max) at (0,10) {\usebox{\ttmaxbox}}; 
  \node (2323131312) at (-2,8) {\usebox{\tthtthothothotbox}};
  \node (23231313) at (-4,6) {\usebox{\tthtthothothbox}};
  \node (232313) at (-4,4) {\usebox{\tthtthothbox}};
  \node (121213) at (4,4) {\usebox{\ototothbox}};
  \node (2323) at (-4,2) {\usebox{\tthtthbox}};
  \node (1223) at (0,4) {\usebox{\ottthbox}};
  \node (1212) at (4,2) {\usebox{\ototbox}};
  \node (12131213) at (2,8) {\usebox{\otothotothbox}};
  \node (122312) at (2,6) {\usebox{\ottthotbox}};
  \node (12121313) at (4,6) {\usebox{\ototothothbox}}; 
  \node (2313) at (-2,2) {\usebox{\tthothbox}};
  \node (23) at (-2,0) {\usebox{\tthbox}};
  \node (12) at (2,0) {\usebox{\otbox}};
  \node (min) at (0,-2) {\usebox{\ttminbox}};
  \draw (min) --node [below left,blue]{$2$} (23);
  \draw (23) --node [right,blue]{$1$} (2313);
  \draw (2313) --node [above left,blue]{$1$} (1223);
  \draw (min) --node [below right,blue]{$1$} (12);
  \draw (12) --node [above right,blue]{$2$} (1223);
  \draw (23) --node [below left,blue]{$2$} (2323);
  \draw (2323) --node [left,blue]{$1$} (232313);
  \draw (2313) --node [above right,blue]{$2$} (232313);
  \draw (232313) --node [left,blue]{$1$} (23231313);
  \draw (23231313) --node [above left,blue]{$1$} (2323131312);
  \draw (1223) --node [right,blue]{$2$} (2323131312);
  \draw (2323131312) --node [above left,blue]{$1$} (max);
  \draw (1223) --node [below right,blue]{$1$} (122312);
  \draw (122312) --node [left,blue]{$1$} (12131213);
  \draw (12131213) --node [above right,blue]{$2$} (max);
  \draw (12) --node [below right,blue]{$1$} (1212);
  \draw (1212) --node [right,blue]{$1$} (121213);
  \draw (121213) --node [below left,blue]{$2$} (122312);
  \draw (121213) --node [right,blue]{$1$} (12121313);
  \draw (12121313) --node [above right,blue]{$2$} (12131213);
\end{tikzpicture}}}
    \caption{Examples of {\swo}. The labeling is our SB-labeling in \cref{thm:sblabelingsection}.}
    \label{fig:swoexamples}
\end{figure}

Below in \cref{ex:diffplantr}, we illustrate \cref{def:multinvsets} and \cref{prop:invchartrees}. We use subscripts on pairs $(y,x)$ to indicate their multiplicity in a multi-inversion set.

\begin{example}\label{ex:diffplantr} Illustrating \cref{def:multinvsets} and \cref{prop:invchartrees}, we take \[T_1 = \usebox{\otbox} \text{ and } T_2 = \usebox{\otothotothbox} \] and observe that $\inv{T_1}= \{ (2,1)_1\}$ and $\inv{T_2} = \{ (2,1)_2, (3,1)_2, (3,2)_1\}$. Thus, $\inv{T_1}\subseteq \inv{T_2}$ and $\inv{T_2} - \inv{T_1} = \{ (2,1)_1, (3,1)_2, (3,2)_1\}$. Now we note that while $\inv{T_1}=\{ (2,1)_1\}$ is transitive, $I=\{ (2,1)_1,(3,2)_1\}$ is not transitive because $\card{I}{3}{2}=1$ while $\card{I}{2}{1}=1\neq 0$ and $\card{I}{3}{1}=0<\card{I}{3}{2}$. Similarly, $\inv{T_2} = \{ (2,1)_2, (3,1)_2, (3,2)_1\}$ is planar while $J= \{(2,1)_1,(3,1)_1\}$ is not planar because $\card{J}{3}{1}=1$, but $\card{J}{2}{1}=1 \neq 2 = s(2)$ and $\card{J}{3}{2}=0<\card{J}{3}{1}$. 
\end{example}

\begin{remark} \label{rmk:weakorder}
Taking $s=(1,\dots,1)$, $s$-weak order is isomorphic to weak order on the symmetric group $S_{l(s)}$. 
\end{remark}

The following operations on multi-inversion sets are necessary to formulate the join in {\swo} which we will use in the course of our proofs. We give examples of these operations in \cref{ex:unpltrclosure} below.
\begin{itemize}
    \item For weak composition $s$ and multi-inversion sets $I$ and $J$ satisfying $\card{I}{y}{x},\card{J}{y}{x}\leq s(y)$ for all $1\leq x<y \leq n$, the \textbf{union of I and J} is the smallest multi-inversion set by inclusion $\boldsymbol{I\cup J}$ such that $I,J \subseteq I\cup J$, that is $\card{I\cup J}{y}{x} = \max\set{\card{I}{y}{x},\card{J}{y}{x}}$ for all $1\leq x<y\leq n$. Also, the \textbf{sum of I and J} is the multi-inversion set $\boldsymbol{I+J}$ with $\card{I+J}{y}{x} = \min\set{\card{I}{y}{x}+\card{J}{y}{x}, s(y)}$ for all $1\leq x<y\leq n$. If $J=\set{(b,a)}$, we write $\boldsymbol{I+(b,a)}$ for $I+J$.
    
    \item The \textbf{transitive closure}, denoted $\boldsymbol{I^{tc}}$, of a multi-inversion set $I$ is the smallest transitive multi-inversion set, in terms of inclusion, containing $I$.
\end{itemize}

\begin{theorem}\label{thm:latticejoin}\cite[Theorem 2.6]{swkordceballospons2019}
For any weak composition $s$, the $s$-weak order on $s$-decreasing trees is a lattice. The join of two $s$-decreasing trees $T$ and $Z$ is determined by
$$\inv{T \vee Z} = \tcu{\inv{T}}{\inv{Z}}.$$
\end{theorem}

\begin{example}\label{ex:unpltrclosure} This example illustrates the union and sum of multi-inversion sets as well as the transitive closure. Letting $T_1$ be the same {\sdt} as in \cref{ex:diffplantr}, $\inv{T_1}= \{ (2,1)_1\}$. Now $\inv{T_1}\cup \inv{T_1}= \{(2,1)_1\}$ while $\inv{T_1} + \inv{T_1}= \{(2,1)_2\}$. In \cref{ex:diffplantr}, we saw that the multi-inversion set $\{(2,1)_1,(3,2)_1\}$, which is also $\inv{T_1}+(3,2)$, is not transitive. From our observations in \cref{ex:diffplantr}, to satisfy the definition of transitivity in \cref{prop:invchartrees}, $\{(2,1)_1,(3,2)_1\}^{tc}$ must contain $(3,1)$ with multiplicity at least $1$. Thus, $\{(2,1)_1,(3,2)_1\}^{tc}=\{(2,1)_1,(3,1)_1,\newline (3,2)_1\}$. We can check that this is the multi-inversion set of one of the two {\sdt}s covering $T_1$ in (c) of \cref{fig:swoexamples}.
\end{example}

The cover relations in $s$-weak order are characterized as a certain type of operations known as tree rotations. We use this characterization heavily in our proofs. We first need a notion of an ascent in an \sdt. In the case $s=(1,\dots, 1)$, this notion corresponds to the definition of ascents for permutations. Examples of tree ascents of the {\sdt} in \cref{fig:sdtexample} are given in \cref{ex:treeascentsfromfig1}.

\begin{definition}\label{def:treeascent}\cite[Section 2.2]{swkordceballospons2019}
Let $T$ be an $s$-decreasing tree and $1\leq a<b \leq n$. The pair $(a,b)$ is a \textbf{tree ascent} of $T$ if the following hold:
\begin{itemize}
    \item [(i)] $a\in T^b_i$ for some $0\leq i <s(b)$,
    \item [(ii)] if $a \in T^e_j$ for any $a<e<b$, then $j=s(e)$,
    \item [(iii)] if $s(a)>0$, then $T^a_{s(a)}$ is a leaf, that is, $T^a_{s(a)}$ contains no internal vertices.
\end{itemize}
\end{definition}

\begin{example}\label{ex:treeascentsfromfig1}
The tree ascents of the {\sdt} in (a) of \cref{fig:sdtexample} are as follows: $\{(1,4),(2,4),(3,4),(4,5),(5,9),(6,7),(7,8)\}$.
\end{example}

\begin{remark}\label{rmk:noascents0tops}
If $s(b)=0$, then $(a,b)$ with $a<b$ is not a tree ascent of any \sdt. This would contradict (i) of \cref{def:treeascent}.
\end{remark}

\begin{remark}\label{rmk:nosamebottomascents}
An \sdt, $T$, cannot have tree ascents $(a,b)$ and $(a,c)$ with $b\neq c$. This would contradict condition (ii) of \cref{def:treeascent} as either $a<b<c$ or $a<c<b$ while $a\not \in T^b_{s(b)}, T^c_{s(c)}$ by condition (i) of \cref{def:treeascent}. We note that this implies that given an element $c\in [n]$ there is at most one $d\in [n]$ such that $(c,d)$ is a tree ascent of $T$. Further, whenever $(a,b)$ and $(c,d)$ are distinct tree ascents of $T$, we may assume $a<c$. We make this assumption throughout our proofs.
\end{remark}

\begin{remark}\label{rmk:ascentsubtree}
We observe that by \cref{rmk:biggercontain}, conditions (i) and (ii) of \cref{def:treeascent} together are equivalent to $a\in  {_RT^b_i}$ for some $0\leq i <s(b)$. The $i$th rightmost subtree of $b$ in $T$ ${_RT^b_i}$ is defined at the beginning of \cref{sec:bkgdswo}.
\end{remark}

\begin{definition}\label{def:givesstreestrot}\cite[Section 2.2]{swkordceballospons2019}
Let $T$ be an $s$-decreasing tree with tree ascent $(a,b)$. Then $\tcp{\inv{T}}{(b,a)}$ is an $s$-tree inversion set. We call the {\sdt} $Z$ defined by $\inv{Z} = \tcp{\inv{T}}{(b,a)}$ the \textbf{$s$-tree rotation of $T$ along $(a,b)$}. We denote this by $\boldsymbol{\rot{T}{Z}{(a,b)}}$.
\end{definition}

Ceballos and Pons characterized cover relations in {\swo} with the following theorem.
 
\begin{theorem}\label{thm:covers}\cite[Theorem 2.7]{swkordceballospons2019}
Let $T$ and $Z$ be \sdt{s}. Then $T\precdot Z$ if and only if there is a unique pair $(a,b)$ which is a tree ascent of $T$ such that $\rot{T}{Z}{(a,b)}$.
\end{theorem}

\begin{remark}\label{lem:1doesntmatter}
$s(1)$ does not change the isomorphism type of $s$-weak order because no tree ascent of an {\sdt} may have larger element $1$.
\end{remark}

\begin{remark}\label{rmk:treepicture}
We describe an $s$-tree rotation in terms of an operation on the trees themselves. This is illustrated in \cref{fig:genstreerotab}. Suppose $(a,b)$ is a tree ascent of $T$ and $\rot{T}{Z}{(a,b)}$. Then $a\in {_RT^b_j}$ for some $j<s(b)$. Let $g$ be the parent of $a$ so $a\in T^g_{s(g)}$ and $g\in T^b_j$ or $g=b$ and is the $j$th child of $b$. Let $m$ be the smallest element of $_LT^b_{j+1}$ which is still larger than $a$. It is possible $m=b$. Then $Z$ is the same as $T$ except for the following changes: $Z^g_{s(g)}=T^a_0$ if $g\neq b$ and $Z^b_j=T^a_0$ if $g=b$ instead of $T^a$, $Z^a_i= T^a_i$ for $0<i<s(a)$ if $s(a)>0$, $Z^a_{s(a)} = T^m_0$ if $m\neq b$ and $Z^a_{s(a)} = T^b_{j+1}$ if $m=b$, $Z^a_0$ is a leaf is a leaf if $s(a)>0$, and $Z^m_0=Z^a$ if $m\neq b$ and $Z^b_{j+1}=Z^a$ if $m=b$.
\end{remark}

\begin{figure}[H]
\centering
\subfigure{
\scalebox{.8}{
\begin{tikzpicture}[very thick]
  \node [style={draw,circle}]{b}
    child { node[style={draw,circle}] {e} 
        child { node[xshift=-10] {$T^e_{0,\dots,s(e)-1}$} }
        child[grow=down]{ {} 
            child[grow=down,dashed]{ node[solid,style={draw,circle}] {g} 
                child[solid]{ node[xshift=-15] {$T^g_{0,\dots,s(g)-1}$} }
                child[solid,grow=down] { node [style={draw,circle}] {a} 
                    child { node {$T^a_0$} }
                        child[grow=down]{ node {$T^a_{1,\dots,s(a)-1}$} }
                        child { node[style={draw,circle},xshift=-10] {} } 
                }
            }
        }
    }
    child { node[style={draw,circle}] {f} 
        child[grow=down] { {}
            child[dashed,grow=down] { node[solid,style={draw,circle}] {$m$}
                child[solid,grow=down] { node {$T^m_0$} }
                child[solid] { node[xshift=15] {$T^{m}_{1,\dots, s(m)}$} } 
            } 
        }
        child { node[xshift=10] {$T^f_{1,\dots,s(f)}$} } 
    };
\end{tikzpicture}
}}
\subfigure{
\begin{tikzpicture}[very thick]
\draw [thick] [->] (2,3) -- (3,3);
\node at (2.5,3) [above] {$(a,b)$};
\draw [white] (2,1) -- (3,1);
\end{tikzpicture}
}
\subfigure{
\scalebox{.8}{
\begin{tikzpicture}[very thick]
  \node [style={draw,circle}]{b}
    child { node[style={draw,circle}] {e} 
        child { node[xshift=-10] {$T^e_{0,\dots,s(e)-1}$} }
        child[grow=down]{ {} 
            child[grow=down,dashed]{ node[solid,style={draw,circle}] {g} 
                child[solid]{ node[xshift=-15] {$T^g_{0,\dots,s(g)-1}$} }
                child[solid,grow=down] { node {$T^a_0$} }
            }
        }
    }
    child { node[style={draw,circle}] {f} 
        child[grow=down] { {}
            child[dashed,grow=down] { node[solid,style={draw,circle}] {$m$}
                child[solid,grow=down] { node [style={draw,circle}] {a} 
                    child { node[style={draw,circle}] {} }
                        child[grow=down]{ node {$T^a_{1,\dots,s(a)-1}$} }
                        child { node {$T^m_0$} } 
                }
                child[solid] { node[xshift=15] {$T^{m}_{1,\dots, s(m)}$} } 
            } 
        }
        child { node[xshift=10] {$T^f_{1,\dots,s(f)}$} } 
    };
\end{tikzpicture}
}}  
\caption{Illustration of the $s$-tree rotation along the tree ascent $(a,b)$.}
\label{fig:genstreerotab}
\end{figure}
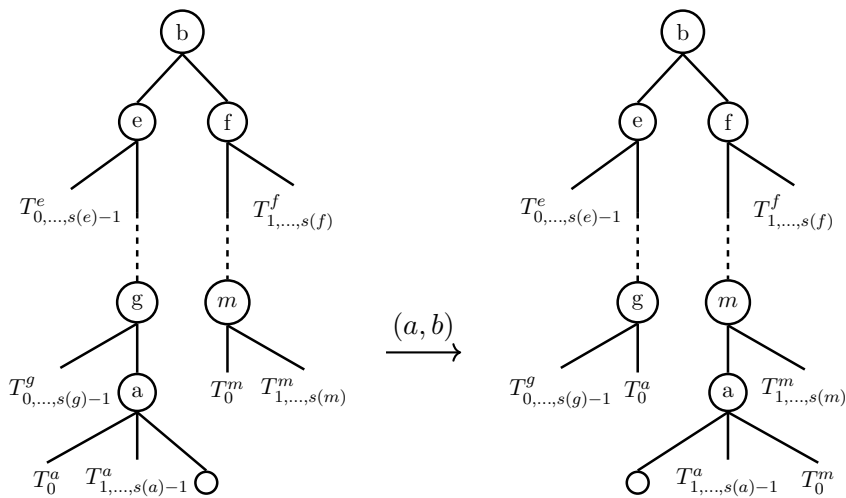

\begin{remark}
One might wonder if {\swo} is a Cambrian lattice of some finite Coxeter group. Cambrian lattices were defined by Reading in \cite{camblatticesreading2006} as certain lattice quotients of weak order. However, from {\swo} with $s=(0,0,2)$ (see \cref{fig:swoexamples}) we observe that {\swo} is not generally a Cambrian lattice of a finite Coxeter group. The Cambrian lattices of a finite Coxeter group $W$ all have order the Coxeter Catalan number $Cat(W)$. The only $W$ with $Cat(W)=9$ is the dihedral group $I_2(7)$ see \cite{latticetheory}. However, {\swo} with $s=(0,0,2)$ has largest anti-chain of cardinality 3 while the largest anti-chain in a Cambrian lattice of $I_2(7)$ has cardinality at most 2.
\end{remark}
\end{subsection}

\begin{subsection}{Background on the $s$-Tamari lattice}\label{sec:bkgdstam}

The Tamari lattice is the sublattice of weak order on permutations generated by the $231$-avoiding permutations. Similarly, the $s$-Tamari lattice is the sublattice of {\swo} generated by certain {\sdt}s. 

\begin{definition}\label{def:stamaritree}\cite[Definition 3.1]{swkordceballospons2019}
An {\sdt} $T$ is called an \textbf{$s$-Tamari tree} if for any $a<b<c$, $\card{T}{c}{a} \leq \card{T}{c}{b}$ where $\card{T}{c}{a}$ is as defined in \cref{def:treeinversions}. That is, all of the vertex labels in $T^c_i$ are smaller than all of the vertex labels in $T^c_j$ for $i<j$. The multi-inversion set of an $s$-Tamari tree is called an $s$-Tamari inversion set.
\end{definition}

We denote the partial order on $s$-Tamari trees induced by {\swo} by $\boldsymbol{\preceq_{Tam}}$. Similarly, a subscript $Tam$ will be used to denote objects in the $s$-Tamari lattice. For instance, $[T,Z]_{Tam}$ is the closed interval from $T$ to $Z$ in the $s$-Tamari lattice.

\begin{theorem}\label{thm:stamarilattice}\cite[Theorem 3.2]{swkordceballospons2019}
The collection of $s$-Tamari trees forms a sublattice of {\swo}, called the $s$-Tamari lattice.
\end{theorem}

\begin{remark}\label{rmk:tamarilattice}
Taking $s=(1,\dots,1)$, the $s$-Tamari lattice is isomorphic to the classical Tamari lattice on $l(s)$.
\end{remark}

Similarly to {\swo}, there is a notion of ascent for {\stt}s and cover relations in the $s$-Tamari lattice are characterized as certain tree rotations along these ascents. 
For $a<b$, we say that $(a,b)$ is a \textbf{Tamari tree ascent} of $T$ if $a$ is a non-right most child of $b$, that is, $a$ is a direct descendant of $b$ and $\card{T}{b}{a}<s(b)$. Note that in the {\staml}, $T^a_{s(a)}$ need not be a leaf for some $(a,b)$ to be a Tamari tree ascent. We denote cover relations in the $s$-Tamari lattice by $\precdot_{Tam}$.

\begin{theorem}\label{thm:stamaricovers}\cite[Section 3.1]{swkordceballospons2019}
Let $T$ be an $s$-Tamari tree and let $(a,b)$ be a Tamari tree ascent of $T$. Then $\tcp{\inv{T}}{(b,a)}$ is an $s$-Tamari inversion set. Let $Z$ be the {\stt} such that $\inv{Z}=\tcp{\inv{T}}{(b,a)}$. We say $Z$ is the \textbf{$s$-Tamari rotation of $T$ along $(a,b)$} and write $\boldsymbol{\trot{T}{Z}{(a,b)}}$. Moreover, $T\precdot_{Tam} Z$ if and only if there is a unique Tamari tree ascent $(a,b)$ of $T$ such that $\trot{T}{Z}{(a,b)}$.
\end{theorem}

An $s$-Tamari rotation is essentially the same as an {\str} except that the smaller element of the Tamari tree ascent may have right descendants and those right descendants are moved with along with $a$ if $s(a)>0$. An $s$-Tamari rotation is illustrated in \cref{fig:gentamarirotab}.

\begin{remark}\label{rmk:stamaritreepic}
Similarly to {\str}s, we describe $s$-Tamari rotations in terms of an operation on the trees themselves. Suppose that $(a,b)$ is a Tamari tree ascent of $T$ and $\trot{T}{Z}{(a,b)}$. Then $a\in T^b_j$ for some $j<s(b)$ and $a$ is a child of $b$. Recall that every labeled vertex of $T^b_{j+1}$ is greater than $a$ since $T$ is an $s$-Tamari tree. Let $m$ be the smallest labeled vertex of $_LT^b_{j+1}$. Then $Z$ is the same as $T$ except for the following: $Z^b_j=T^a_0$ instead of $T^a$, $Z^a_i=T^a_i$ for $0<i\leq s(a)$ if $s(a)>0$, $Z^a_0$ is a leaf, $Z^m_0=Z^a$. 
\end{remark}

\begin{remark}\label{rmk:nosamebottomtamascents}
An $s$-Tamari tree $T$ cannot have Tamari tree ascents $(a,b)$ and $(a,c)$ with $b\neq c$. This follows from the fact that in a rooted tree, every non-root node has exactly one parent. Thus, whenever $(a,b)$ and $(c,d)$ are distinct Tamari tree ascents of $T$, we may assume $a<c$. We make this assumption throughout our proofs.
\end{remark}

\begin{figure}[H]
\centering
\subfigure{
\scalebox{.8}{
\begin{tikzpicture}[very thick]
  \node [style={draw,circle}]{b}
    child { node[style={draw,circle}] {a} 
        child { node[xshift=-10] {$T^a_{0}$ } }
        child[grow=down]{ node {$T^a_{1,\dots, s(a)}$ } } 
        }
    child { node[style={draw,circle}] {f} 
        child[grow=down] { {}
            child[dashed,grow=down] { node[solid,style={draw,circle}] {$m$}
                child[solid,grow=down] { node[style={draw,circle}] {} }
                child[solid] { node[xshift=15] {$T^{m}_{1,\dots, s(m)}$} } 
            } 
        }
        child { node[xshift=10] {$T^f_{1,\dots,s(f)}$} } 
    };
\end{tikzpicture}
}}
\subfigure{
\begin{tikzpicture}[very thick]
\draw [thick] [->] (2,3) -- (3,3);
\node at (2.5,3) [above] {$Tam(a,b)$};
\draw [white] (2,1) -- (3,1);
\end{tikzpicture}
}
\subfigure{
\scalebox{.8}{
\begin{tikzpicture}[very thick]
 \node [style={draw,circle}] {b}
    child { node {$T^a_0$} }
    child { node[style={draw,circle}] {f} 
        child[grow=down] { {}
            child[dashed,grow=down] { node[solid,style={draw,circle}] {$m$}
                child[solid,grow=down] { node[style={draw,circle}] {a}
                    child[solid,grow=down] { node[style={draw,circle}] {} }
                    child{ node[xshift=10] {$T^a_{1,\dots,s(a)}$} }
                    }
                child[solid] { node[xshift=15] {$T^{m}_{1,\dots, s(m)}$} } 
            } 
        }
        child { node[xshift=10] {$T^f_{1,\dots,s(f)}$} } 
    };
\end{tikzpicture}
}}  
\caption{$s$-Tamari rotation along the Tamari tree ascent $(a,b)$.}
\label{fig:gentamarirotab}
\end{figure}
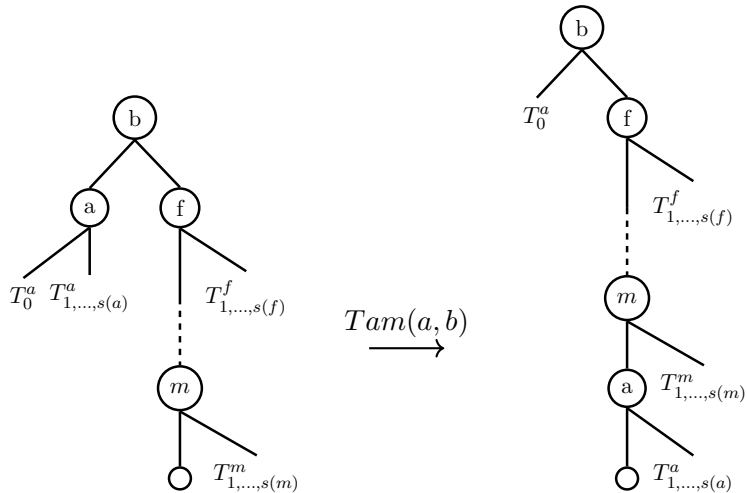

\end{subsection}

\begin{subsection}{Background on SB-labelings} \label{sec:sblabels}
Hersh and M{\'e}sz{\'a}ros developed the notion of an SB-labeling in \cite{sblabelhershmeszaros2017} to show when certain lattices have open intervals which are homotopy balls or spheres.

\begin{definition}\label{def:sblabeling}\cite[Definition 3.4]{sblabelhershmeszaros2017}
An \textbf{SB-labeling} is an edge labeling $\lambda$ on a finite lattice $L$ satisfying the following conditions for each $u, v,w \in L$ such that $v$ and $w$ are distinct elements which each cover $u$:
\begin{enumerate}
    \item [(i)] $\lambda(u,v)\neq\lambda(u,w)$
    \item [(ii)] Each saturated chain from $u$ to $v\vee w$ uses both of these labels $\lambda(u, v)$ and $\lambda(u,w)$ a positive number of times.
    \item [(iii)] None of the saturated chains from $u$ to $v\vee w$ use any other labels besides $\lambda(u, v)$ and $\lambda(u,w)$.
\end{enumerate}
\end{definition}

One of Hersh and M{\'e}sz{\'a}ros' main theorems in \cite{sblabelhershmeszaros2017} is the following characterization of the homotopy types of intervals in a lattice which admits an SB-labeling.

\begin{theorem}\label{thm:sbthm}\cite[Theorem 3.7]{sblabelhershmeszaros2017}
If $L$ is a finite lattice which admits an SB-labeling, then each open interval $(u,v)$ in $L$ is homotopy equivalent to a ball or a sphere of some dimension. Moreover, $\Delta(u, v)$ is homotopy equivalent to a sphere if and only if $v$ is a join of atoms of $[u,v]$, in which case it is homotopy equivalent to a sphere $S^{d-2}$ where $d$ is the number of atoms in $[u,v]$.
\end{theorem}

We will use this theorem to draw our topological conclusions.

\end{subsection}

\end{section}

\begin{section}{An SB-labeling of $s$-weak order} \label{sec:proofs}

In this section, we prove a series of lemmas on \sdt{s} and multi-inversion sets which we then use to prove that the following edge labeling of {\swo} is an SB-labeling as \cref{thm:sblabeling}. We introduce many of the lemmas with a short more intuitive description of the lemma and its proof as well as give reference to an example. In that spirit, we label a cover relation in {\swo} by taking the unique tree ascent (pair of distinct labeled vertices) corresponding to the cover relation by \cref{thm:covers} and use the smaller of the two elements of the tree ascent as the label, that is we label cover relations by the label of the root vertex of the subtree moved to achieve the cover relation. \cref{fig:swoexamples} includes two examples of our labeling of {\swo}.

\begin{definition}\label{thm:sblabelingsection}
Let $T\precdot Z$ be a cover relation in {\swo}. Let $\rot{T}{Z}{(a,b)}$ be the {\str} of $T$ along the unique tree ascent $(a,b)$ associated to $T\precdot Z$ by \cref{thm:covers}. Define $\lambda$ to be the edge labeling of {\swo} given by $\lambda(T, Z)= a$. 

The notion of tree ascent is defined in \cref{def:treeascent}. The notation $\rot{T}{Z}{(a,b)}$ and corresponding notion of $s$-tree rotation are given in \cref{def:givesstreestrot} and \cref{rmk:treepicture}.
\end{definition}

\begin{remark}
In the case $s=(1,\dots,1)$, the SB-labeling of \cref{thm:sblabelingsection} gives an SB-labeling of weak order on $S_{l(s)}$. Our labeling is distinct from the one given for symmetric groups by Hersh and M{\'e}sz{\'a}ros in \cite{sblabelhershmeszaros2017}.
\end{remark}

The main point of our proof that \cref{thm:sblabelingsection} is an SB-labeling of {\swo} is showing that for any $T\precdot Z,Q$, the Hasse diagram of the interval $[T,Z\vee Q]$ is a diamond, a pentagon, or a hexagon. Examples of all three types of such intervals, as well as the underlying reasons they occur which have to do with relationships between tree ascents that are explained in later lemmas, can be seen in \cref{fig:swoexamples}. In particular, $[T,Z\vee Q]$ has precisely two maximal chains. Then we verify that, in any case, the labeling on the two maximal chains satisfies \cref{def:sblabeling}. Many of our proofs are easier with \cref{fig:genstreerotab} and \cref{rmk:treepicture} in mind so it is worth a few moments to internalize those.

The following proposition restricts the possible tree ascents of an {\sdt}. In particular, if $(a,b)$ is a tree ascent of some {\sdt} $T$ with $s(a)>0$, then no labeled vertices of $T^a$ besides $a$ can form a tree ascent with $b$. For instance, $(5,9)$ is a tree ascent of the {\sdt} in \cref{fig:sdtexample}, but no vertex below $5$ forms a tree ascent with $9$ because the rightmost child of $5$ must be a leaf. We use this to characterize the multi-inversion set of $Z\vee Q$ for any $T\precdot Z,Q$ and to restrict the chains that can appear in $[T, Z\vee Q]$.

\begin{proposition}\label{lem:nolowerascents}
Let $T$ be an {\sdt} and let $1\leq a<b \leq n$ be such that $(a,b)$ is a tree ascent of $T$ with $s(a)>0$. Then no pair of the form $(e,b)$ such that $e\in T^a$ and $e<a$ is a tree ascent of $T$. 
\end{proposition}

\begin{proof}
Let $(a,b)$ be a tree ascent of $T$. Assume $(e,b)$ is also a tree ascent of $T$ with $e\in T^a$ and $e<a$. Then $e\in T^a_{s(a)}$ by (ii) of \cref{def:treeascent} of $(e,b)$ being a tree ascent of $T$ because $e<a<b$. Thus, $T^a_{s(a)}$ is not a leaf.  However, since $s(a)>0$, this contradicts (iii) of \cref{def:treeascent} of $(a,b)$ being a tree ascent of $T$. Thus, such a pair $(e,b)$ is not a tree ascent of $T$.
\end{proof}

\begin{remark}\label{rmk:0lowerascent}
If $s(a)=0$, it is possible that $(a,b)$ and $(e,b)$ for some $e\in T^a$ with $e<a$ are both tree ascents of $T$.

The situation precluded by \cref{lem:nolowerascents} may occur if $s(a)=0$.
\end{remark}

We use the following two definitions to describe $Z\vee Q$ for any $T\precdot Z,Q$ in terms of tree inversion sets.

\begin{definition}\label{def:invsaddedset}
Let $T$ be a {\sdt} and let $1\leq a<b \leq n$ be such that $(a,b)$ is a tree ascent of $T$. Let $Z$ be the {\sdt} obtained by $\rot{T}{Z}{(a,b)}$. Define \textbf{the set of inversions added by the {\str} along $\boldsymbol{(a,b)}$}, denoted $\boldsymbol{A_T(a,b)}$, by \[A_T(a,b) =\sett{(f,e)}{\card{Z}{f}{e}>\card{T}{f}{e}}. \]
\end{definition}

\begin{definition}\label{def:tertaddedinvs}
Let $T$ be an {\sdt}. Let $(a,b)$ and $(c,d)$ be tree ascents of $T$ with $a<c$. We note that $b$ and $d$ are determined by \cref{rmk:nosamebottomascents} once we know $a$ and $c$ are each the smaller element of a tree ascent. Define the following set valued function: \[F_T(a,c) =\begin{cases} \sett{(d,e)}{e\in T^a\setminus 0} & \text{ if } b=c \text{ and }a\in T^c_0 \\
\\
\emptyset & \text{otherwise}
\end{cases} \]
\end{definition}

\begin{example}\label{ex:exaddedinvs}
Let $T$ be the {\sdt} in \cref{fig:sdtexample}. As we saw in \cref{ex:treeascentsfromfig1}, $(5,9)$ and $(4,5)$ are both tree ascents of $T$. Also, $4\in T^5_0$. If we perform the {\str} of $T$ along $(5,9)$ using \cref{rmk:treepicture}, we observe that $A_T(5,9)=\{(9,5)\}$ and $A_T(4,5)=\{(5,1),(5,2),(5,4)\}$. Also, by definition $F_T(4,5)=\{(9,1),(9,2),(9,4)\}$.
\end{example}

In the next proposition, we explicitly compute the tree inversions added by an {\str}, that is, $A_T(a,b)$ from \cref{def:invsaddedset}. The proposition can be verified on \cref{ex:exaddedinvs} above.

\begin{proposition}\label{lem:inversionsadded}
Let $T$ be an {\sdt} and let $1\leq a<b \leq n$ be such that $(a,b)$ is a tree ascent of $T$. Suppose $\rot{T}{Z}{(a,b)}$. Then for $1\leq e < f\leq n$, $(f,e) \in A_T(a,b)$ if and only if $f=b$ and $e\in T^a\setminus 0$ in which case $$\#_Z(f,e)=\#_T(f,e) + 1.$$ The notation $\card{Z}{f}{e}$ and the corresponding notion of cardinality are given in \cref{def:treeinversions}.
\end{proposition}

\begin{proof}
First, we note that if $e\in T^a\setminus 0$ and $e<a$, then $s(a)>0$. Thus, $T^a_{s(a)}$ is a leaf by condition (iii) of \cref{def:treeascent} of $(a,b)$ being a tree ascent of $T$. Hence, for any $e\in T^a\setminus 0$, $e\not \in T^a_{s(a)}$. Then both parts of the statement follow from \cref{rmk:treepicture} by considering the only subtrees that change in an {\str} (see \cref{fig:genstreerotab}).
\end{proof}

A particularly simple case of \cref{lem:inversionsadded} is when the smaller element of a tree ascent has only a single child.

\begin{corollary}\label{cor:invsadded0}
If $(a,b)$ is a tree ascent of $T$ with $s(a)=0$ and $\rot{T}{Z}{(a,b)}$, then $A_T(a,b)=\set{(b,a)}$.
\end{corollary}

The subsequent lemma essentially shows that the sets of inversions added by {\str}s along distinct tree ascents are disjoint. This is illustrated by \cref{ex:exaddedinvs} where the particular sets of inversions added are pairwise disjoint. We use this lemma in the proof of one of two different upcoming characterizations of $Z\vee Q$ for any $T\precdot Z,Q$. The proof relies on the restrictions on tree ascents from \cref{cor:invsadded0} and on our characterization of tree inversions added by an {\str} from \cref{lem:inversionsadded}.

\begin{lemma}\label{lem:nocommoninvsadded}
Let $T$ be an {\sdt}. Let $1\leq a<b \leq n$ and $1\leq c<d \leq n$ be such that $(a,b)$ and $(c,d)$ are tree ascents of $T$ with $a<c$. Then $A_T(a,b)$, $A_T(c,d)$, and $F_T(a,c)$ are pairwise disjoint. 

The notation $A_T(a,b)$ and $F_T(a,c)$ are given in \cref{def:invsaddedset} and \cref{def:tertaddedinvs}, respectively.
\end{lemma}

\begin{proof}
We assume seeking contradiction that there is some $(f,e)\in A_T(a,b) \cap A_T(c,d)$. Then by \cref{lem:inversionsadded}, $f=b=d$ and $e\in T^a, T^c$. Now by \cref{def:treeascent} of $(a,b)$ and $(c,d)$ being tree ascents of $T$, $a,c\in T^b$. Then, by the fact that $e$ is only below one child of $b$ in $T$ and by \cref{rmk:biggercontain}, $a,c\in T^b_i$. Then since $(a,b)$ and $(c,b)$ are both tree ascents of $T$, $a,c\in {_RT^b_i}$ by \cref{rmk:ascentsubtree}. Now by definition of ${_RT^b_i}$, $a\in T^c$. If $s(c)>0$, then $(a,b)$ and $(c,b)$ both being tree ascents of $T$ contradicts \cref{lem:nolowerascents}. Thus, $s(c)=0$. Then by \cref{cor:invsadded0}, $A_T(c,d)=\{(d,c)\}$ so $(f,e)=(d,c)$. But that contradicts \cref{lem:inversionsadded} because $a<c$ so $c\not \in T^a\setminus 0$. 

If $F_T(a,c)\neq \emptyset$, then $b=c\neq d$ and $a\in T^c_0$ by \cref{def:tertaddedinvs}. Thus,  $F_T(a,c)$ is disjoint from $A_T(a,b)$  by \cref{lem:inversionsadded} since $b\neq d$. Also in this case, $F_T(a,c)$ is disjoint from $A_T(c,d)$ by \cref{lem:inversionsadded} because each $e\in T^a\setminus 0$ is also in $T^c_0$.
\end{proof}

Now we have the first of two different descriptions of $Z\vee Q$ for any $T\precdot Z,Q$. The second description is $Z\vee Q$ below. Intuitively, this lemma says we can reach $Z\vee Q$ by first performing the {\str} of $T$ along the tree ascent associated with $Z$ and then the {\str} of $Z$ along the tree ascent associated with $Q$ or vice versa. In reality, we run into situations where the tree ascent of $T$ associated with $Q$ is not actually a tree ascent of $Z$ or vice versa. So this intuitive picture is not always defined. We address those situations with later lemmas. We use this description to establish the desired saturated chains in $[T,Z\vee Q]$, while we use the second description in the proofs that there are no other saturated chains to such a join. We prove this lemma by showing double containment of multi-inversion sets using the definition of transitive closure and our characterization of the tree inversions added by an {\str} from \cref{lem:inversionsadded}.

\begin{lemma}\label{lem:joinoftwoats}
Let $T$ be an {\sdt} and let $1\leq a<b \leq n$ and $1\leq c<d \leq n$ be such that $(a,b)$ and $(c,d)$ are distinct tree ascents of $T$. Suppose $\rot{T}{Z}{(a,b)}$ and $\rot{T}{Q}{(c,d)}$. Then $\inv{Z\vee Q} = \tcp{{\tcp{\inv{T}}{(b,a)}}}{(d,c)}$. Moreover, the order of the pairs in this equality of multi-inversion sets can be reversed.

The notation $(\cdot)^{tc}$ and the corresponding notion of transitive closure are given just prior to \cref{thm:latticejoin}. The notion of containment of multi-inversion sets is given in \cref{def:multinvsets}. The notation $I+J$ and an associated idea of the sum of multi-inversion sets are given just after \cref{ex:diffplantr}.
\end{lemma}

\begin{proof}
First, by \cref{thm:latticejoin}, $\inv{Z\vee Q} = (\inv{Z}\cup \inv{Q})^{tc}.$ Let $I=\inv{Z}\cup \inv{Q}$. By definition of transitive closure, to show $$\tcp{{\tcp{\inv{T}}{(b,a)}}}{(d,c)} = I^{tc}$$ it suffices to show that $\tcp{\inv{T}}{(b,a)}+(d,c) \subseteq I$ and $\inv{Z},\inv{Q}\subseteq \tcp{\inv{T}}{(b,a)}+(d,c)$. We will show the inclusions in that order.

We recall by \cref{def:givesstreestrot} that $\inv{Z}= \tcp{\inv{T}}{(b,a)}$ and $\inv{Q}= \tcp{\inv{T}}{(c,d)}$. By \cref{lem:inversionsadded} and \cref{lem:nocommoninvsadded}, $\card{Z}{d}{c}=\card{T}{d}{c}$ and $\card{Q}{d}{c}=\card{T}{d}{c}+1$ so $\card{I}{d}{c}= \card{T}{d}{c}+1$. Thus, $\tcp{\inv{T}}{(b,a)}+(d,c) \subset I$. On the other hand, $\inv{T}+(d,c)\subset {\tcp{\inv{T}}{(b,a)}}+(d,c)$ since $\inv{T}\subset \tcp{\inv{T}}{(b,a)}$. Thus, $\inv{Q}, \newline \inv{Z} \subset ((\inv{T}+ (b,a))^{tc}+(d,c))^{tc}$. Therefore, $\inv{Z\vee Q} = ((\inv{T}+ \newline  (b,a))^{tc}+(d,c))^{tc}$. Similarly, the tree ascents may appear in the other order, that is $\inv{Z\vee Q} = \tcp{{\tcp{\inv{T}}{(d,c)}}}{(b,a)}$.
\end{proof}

In the next lemma, we begin with distinct tree ascents $(a,b)$ and $(c,d)$ of an {\sdt} $T$ and let $Z$ and $Q$ be the {\str}s of $T$ along those tree ascents, respectively. We characterize when one pair ceases to be a tree ascent of the {\str} along the other pair. Intuitively, $(c,d)$ only stops being a tree ascent of $Z$ if the subtree of $T$ rooted at $a$ is moved to the right most child of $c$ by the rotation. Similarly, $(a,b)$ only stops being a tree ascent of $Q$ if the subtree of $T$ rooted at $a$ is left behind while $b$ is moved by the {\str}. The four possibilities turn out to correspond to different relationships between $(a,b)$ and $(c,d)$ in $T$. These four possibilities end up characterizing the intervals $[T,Z\vee Q]$ which have Hasse diagrams that are diamonds, pentagons, and hexagons. In later lemmas, we will show that in particular, when $(a,b)$ is a tree ascent of $Q$ and $(c,d)$ is a tree ascent of $Z$, $[T,Z\vee Q]$ has Hasse diagram that is a diamond. When exactly one of $(a,b)$ is not a tree ascent of $Q$ or $(c,d)$ is not a tree ascent of $Z$, $[T, Z\vee Q]$ has Hasse diagram which is a pentagon. When both $(a,b)$ is not a tree ascent of $Q$ and $(c,d)$ is not a tree ascent of $Z$, $[T,Z\vee Q]$ has Hasse diagram that is a hexagon. \cref{lem:stopbeingascent} below can be illustrated with the {\sdt} in \cref{fig:sdtexample}. Using \cref{rmk:treepicture}, we can perform the {\str}s of the {\sdt} in \cref{fig:sdtexample} along the following pairs of tree ascents which exemplify the ways a pair of tree ascents can be related and all of the ways one tree ascent can cease to be a tree ascent after the {\str} along the another tree ascent: $(5,9)$ and $(7,8)$, $(2,4)$ and $(3,4)$, $(3,4)$ and $(4,5)$, $(2,4)$ and $(4,5)$, $(4,5)$ and $(5,9)$.

\begin{lemma}\label{lem:stopbeingascent}
Let $T$ be a {\sdt}. Let $1\leq a<b \leq n$ and $1\leq c<d \leq n$ be such that $(a,b)$ and $(c,d)$ are tree ascents of $T$ with $a<c$. Let $\rot{T}{Z}{(a,b)}$ and $\rot{T}{Q}{(c,d)}$. If either of $(a,b)$ is not a tree ascent of $Q$ or $(c,d)$ is not a tree ascent of $Z$, then $b=c$ and $s(c)>0$. Moreover, if $(a,c)$ is not a tree ascent of $Q$, then $a\in T^c_0$. If $(c,d)$ is not a tree ascent of $Z$, then $a\in T^c_{s(c)-1}$.
\end{lemma}

\begin{proof}
We will argue that there are four cases that we must check in more detail for the way in which one of the tree ascents $(a,b)$ or $(c,d)$ can cease to be a tree ascent after the {\str} along the other. We will check these four cases and show that two of them cannot actually occur and that the other two are precisely the conclusions of the lemma. Suppose that either $(a,b)$ is not a tree ascent of $Q$ or $(c,d)$ is not a tree ascent of $Z$. Then after the {\str} along one of $(a,b)$ or $(c,d)$, at least one of the three conditions of \cref{def:treeascent} must be violated by the other pair. 

We begin with two observations with which we show three simpler cases cannot occur leaving us with the four cases mentioned above. First, since $a<c$, \cref{rmk:treepicture} implies the {\str} along $(a,b)$ does not move vertex $c$ or any vertices above $c$ in $T$. Second, {\str}s never decrease the cardinalities of tree inversions by \cref{lem:inversionsadded}.

The first observation shows that condition (i) of \cref{def:treeascent} cannot be violated by $(c,d)$ in $Z$ because the relative positions of $c$ and $d$ in $T$ are not changed by the {\str} along $(a,b)$. The first and second observations together show that condition (ii) of \cref{def:treeascent} cannot be violated by $(c,d)$ in $Z$. This is because the first observation implies that for any $e$ with $c<e<d$, $c\in T^e$ if and only if $c\in Z^e$. By condition (ii) of \cref{def:treeascent} of $(c,d)$ being a tree ascent of $T$, $\card{T}{e}{c}=s(e)$. Then by the second observation, $\card{T}{e}{c}\leq \card{Z}{e}{c}$ so $\card{Z}{e}{c}=s(e)$, which is exactly condition (ii) of \cref{def:treeascent} of $(c,d)$ being a tree ascent of $Z$. Lastly, the second observation shows that condition (ii) of \cref{def:treeascent} cannot be violated by $(a,b)$ in $Q$ in certain cases, namely by any $e$ such that $a<e<b$, $a\in Q^e$, and $a\in T^e$. This is again because condition (ii) of \cref{def:treeascent} of $(a,b)$ being a tree ascent of $T$ implies $\card{T}{e}{a}=s(e)$ and the second observation implies $\card{T}{e}{a}\leq \card{Q}{e}{a}$ so $\card{Q}{e}{a}=s(e)$. The case of $a<e<b$ with $a\in Q^e$, but $a\not\in T^e$ is covered as case (1) below.

Thus, there are four possible cases for how conditions (i), (ii), or (iii) of \cref{def:treeascent} might be violated.

\begin{enumerate}
    \item [(1)] $(a,b)$ is not a tree ascent of $Q$ because (ii) is violated by some $a<e<b$ such that $a\not \in T^e$, but $a\in Q^e_i$ and $i<s(e)$.
    \item [(2)] $(a,b)$ is not a tree ascent of $Q$ because (i) is violated by $\card{Q}{b}{a}=s(b)$,
    \item [(3)] $(a,b)$ is not a tree ascent of $Q$ because (i) is violated by $a\not\in Q^b$,
    \item [(4)] $(c,d)$ is not a tree ascent of $Z$ because (iii) is violated by  $s(c)>0$ and $Z^c_{s(c)}$ is not a leaf.
\end{enumerate}

 We show cases (1) and (2) cannot occur and that cases (3) and (4) give the conclusions of \cref{lem:stopbeingascent}. 

\begin{enumerate}
    \item [(1)] Assume there is some $e$ such that $a<e<b$, $a\not\in T^e$, and $a\in Q^e_i$ with $i<s(e)$. By \cref{rmk:treepicture}, there are two ways that $a$ is below vertex in $Q$ which it was not below in $T$. Either $a\in Q^c_{s(c)}$ or $a\in T^c\setminus 0$. If $a\in Q^c_{s(c)}$, then the only vertex that $a$ is below in $Q$ which it was not below in $T$ is $c$. Thus, $e=c$, but $\card{Q}{c}{a}=s(c)$ so (ii) would not be violated. If $a\in T^c\setminus 0$, then $a\neq c$ implies $s(c)>0$. However, if $c<b$, then $a\in T^c_{s(c)}$ because $(a,b)$ is a tree ascent of $T$. This contradicts $(c,d)$ being a tree ascent of $T$ because $s(c)>0$ and $T^c_{s(c)}$ is not a leaf. If $b\leq c$, then $b\in T^c\setminus 0$ by \cref{rmk:biggercontain}. Then by \cref{rmk:treepicture}, if $e$ has $a\in Q^e$ and $a\not\in T^e$, then $c\in Q^e$ also. Thus, $e\geq c\geq b$ contradicting $e<b$. Thus, this case cannot occur.
    
    \item [(2)] Assume $\card{Q}{b}{a} = s(b)$. Since $(a,b)$ is a tree ascent of $T$, $\card{T}{b}{a}<s(b)$. Thus, $\card{Q}{b}{a}=s(b)$ implies $(b,a)\in A_T(c,d)$ by \cref{lem:inversionsadded}. However, this contradicts \cref{lem:nocommoninvsadded}. Hence, this case cannot occur.
    
    \item [(3)] Suppose $a\not \in Q^b$. We note that \cref{rmk:treepicture} (\cref{fig:genstreerotab}) implies that $a\in T^b$ and $a\not \in Q^b$ if and only if $b=c$ and $a\in T^c_0$ by considering the subtrees which change with the {\str}. Thus, $b=c$ and $a\in T^c_0$. Then since $b=c$ and $(a,b)$ is a tree ascent of $T$, $s(c)>0$ by \cref{rmk:noascents0tops}. This is precisely the first of the two possible conclusions of \cref{lem:stopbeingascent}.
    
    \item [(4)] Suppose $s(c)>0$ and $Z^c_{s(c)}$ is not a leaf. We note that $T^c_{s(c)}$ is a leaf by (iii) of \cref{def:treeascent} of $(c,d)$ being a tree ascent of $T$ since $s(c)>0$. Now \cref{rmk:treepicture} implies that $T^c_{s(c)}$ is a leaf and $Z^c_{s(c)}$ is not a leaf if and only if $c=b$ and $a\in T^c_{s(c)-1}$ again by considering the subtrees which change with the {\str}. Hence, $b=c$ and $a\in T^c_{s(c)-1}$. This is exactly the second possible conclusion of \cref{lem:stopbeingascent}.
\end{enumerate}
\vspace{-5mm}
\end{proof}

\begin{remark}
Assuming the hypotheses of \cref{lem:stopbeingascent}, if $s(c)=0$, condition (iii) of \cref{def:treeascent} cannot be violated by $(c,d)$ in $Z$. In this case, $(c,d)$ will be a tree ascent of $Z$.
\end{remark}

In the following lemma, we give a second description of $Z\vee Q$ for any $T\precdot Z,Q$. We explicitly find the multi-inversion set difference between $\inv{T}$ and $\inv{Z\vee Q}$, in contrast with \cref{lem:joinoftwoats} which was the first description of $Z\vee Q$. Similarly to \cref{lem:joinoftwoats} though, $\inv{T}$ and $\inv{Z\vee Q}$ is obtained from $\inv{T}$ by adding the tree inversions necessary to reach $Z$ from $T$ and then the tree inversions needed to reach $Q$ from $T$ but with a correction of some additional tree inversions if $(a,b)$ is not a tree ascent of $Z$. In practice, this lemma shows the possible pairs that may occur as tree ascents corresponding to cover relations in the interval $[T,Z\vee Q]$. We use this lemma to restrict the chains that can occur in $[T,Z\vee Q]$. To show this lemma, we consider two cases based on relationships between tree ascents from \cref{lem:stopbeingascent}. In the proof of the trickier of the two, we construct one of the saturated chains that can occur in $[T,Z\vee Q]$ corresponding to the case from \cref{lem:stopbeingascent} where $T$ has tree ascents $(a,b)$ and $(c,d)$ with $\rot{T}{Z}{(a,b)}$, $\rot{T}{Q}{(c,d)}$, and $(a,b)$ is not a tree ascent of $Q$. The construction of the chain is illustrated in \cref{fig:sc0pentalongforproof} below. We can also verify the lemma on the {\sdt} in \cref{fig:sdtexample} in the case of the cover relations given by the tree ascents and tree inversions added in \cref{ex:exaddedinvs}.

\begin{lemma}\label{lem:addedinvsetcharact}
Let $T$ be an {\sdt}. Let $1\leq a<b \leq n$ and $1\leq c<d \leq n$ be such that $(a,b)$ and $(c,d)$ are tree ascents of $T$ with $a<c$. Suppose $\rot{T}{Z}{(a,b)}$ and $\rot{T}{Q}{(c,d)}$, then $\inv{Z\vee Q} - \inv{T} = A_T(a,b)\cup A_T(c,d) \cup F_T(a,c)$.

The notation $\inv{\cdot}-\inv{\cdot}$ and the corresponding notion of multi-inversion set difference are defined in \cref{def:multinvsets}. The notations $A_T(\cdot,\cdot)$ and $F_T(\cdot,\cdot)$ are defined in \cref{def:invsaddedset} and \cref{def:tertaddedinvs}. 
\end{lemma}

\begin{proof}
There are two cases. Either $(a,b)$ is a tree ascent of $Q$ or not.

Suppose $(a,b)$ is a tree ascent of $Q$. Then either $b\neq c$ or $a\notin T^c_0$ by \cref{lem:stopbeingascent}. Either way, $F_T(a,c)=\emptyset$ by definition. Then by \cref{lem:joinoftwoats}, $\rot{Q}{Z\vee Q}{(a,b)}$. Thus, by \cref{lem:inversionsadded}, $\inv{Z\vee Q} - \inv{T} = A_T(c,d) \cup A_Q(a,b)$. Again using \cref{lem:inversionsadded}, $A_Q(a,b) = \sett{(b,e)}{e\in Q^a\setminus 0}$. Now since $a<c$, $c\not\in T^a$ and $c\not\in Q^a$. Thus, \cref{rmk:treepicture} implies $Q^a\setminus 0 = T^a\setminus 0$. Hence, $A_Q(a,b) = A_T(a,b)$ so $\inv{Z\vee Q} - \inv{T} = A_T(c,d) \cup A_T(a,b)$.

Next suppose $(a,b)$ is not a tree ascent of $Q$. Then $b=c$, $a\in T^c_0$, and $s(c)>0$ by \cref{lem:stopbeingascent}. We first argue that the multi-inversion set difference between $\inv{Z\vee Q}$ and $\inv{T}$ contains the stated tree inversions. We then produce an {\sdt} $P'$ whose multi-inversion set difference with $\inv{T}$ actually equals the stated tree inversions. Then the lemma holds because the join is the least upper bound, in this context has the smallest multi-inversion set difference with $\inv{T}$ by inclusion. We produce $P'$, which is $Z\vee Q$, in the argument by finding a particular saturated chain starting at $T$.

We first observe that by \cref{lem:inversionsadded} and \cref{lem:nocommoninvsadded}, $A_T(a,b)\cup A_T(c,d) \subseteq \inv{Z\vee Q} - \inv{T}$. Next we show that by transitivity $F_T(a,c) \subset \inv{Z\vee Q} - \inv{T}$. It suffices to show that $\card{Z\vee Q}{d}{e} \geq \card{T}{d}{e}+1$ for all $e\in T^a\setminus 0$. To show this inequality we first note that since $b=c$, $e\in Z^c_1$ for all $e\in T^a\setminus 0$ by \cref{rmk:treepicture}. Thus, $\card{Z\vee Q}{c}{e} \geq 1$ for all $e\in T^a\setminus 0$. Now for any such $e\in T^a\setminus 0$, $e<c<d$ so by transitivity $\card{Z\vee Q}{d}{e} \geq \card{Z\vee Q}{d}{c}$. Next we observe that by \cref{lem:inversionsadded} and the fact that $Q\preceq Z\vee Q$, $\card{Z\vee Q}{d}{c}\geq \card{T}{d}{c}+1$. Lastly, we note that since $a\in T^c$, $\card{T}{d}{e}=\card{T}{d}{c}$ for all $e\in T^a\setminus 0$. Thus, $\card{Z\vee Q}{d}{e} \geq \card{T}{d}{e}+1$.

\begin{figure}[H]
\centering
\subfigure{
\scalebox{.6}{
\begin{tikzpicture}[very thick]
  \node [style={draw,circle}]{d}
    child { node[style={draw,circle}] {e} 
        child { node[xshift=-10]    {$T^e_{0,\dots,s(e)-1}$} }
        child[grow=down]{ {} 
            child[grow=down,dashed]{ node[solid,style={draw,circle}] {g}
                child[solid]{ node[xshift=-15] {$T^g_{0,\dots,s(g)-1}$} }
                child[solid,grow=down] { node[style={draw,circle}] {c} 
                    child { node[style={draw,circle},xshift=-10] {k}
                        child { node[xshift=-10] {$T^b_{0,\dots,s(b)-1}$} }
                        child[grow=down] { {} 
                            child[dashed]{ node[solid,style={draw,circle}] {h}
                                child[solid]{ node[xshift=-20] {$T^h_{0,\dots,s(h)-1}$} }
                                child[solid,grow=down] { node [style={draw,circle}] {a} 
                                    child { node {$T^a_0$} }
                                    child[grow=down]{ node {$T^a_{1,\dots,s(a)-1}$} }
                                    child { node[style={draw,circle},xshift=-10] {} } 
                                }
                            }
                        }    
                    }
                child[solid,xshift=-30]{ {}
                    child[grow=down,dashed] { node[solid,style={draw,circle}] {$m_1$}
                        child[grow=down,solid] { node {$T^{m_1}_0$} } 
                        child[solid] { node[xshift=20] {$T^{m_1}_{1,\dots, s(m_1)}$} }  
                        }
                    }    
                child { node[xshift=-30] {$T^c_{2,\dots,s(c)-2}$ } }
                child { node[style={draw,circle},xshift=-40] {} }
                } 
            }
        }
    }
    child { node[style={draw,circle}] {f} 
        child[grow=down] { {}
            child[grow=down] { {}
            child[dashed,grow=down] { node[solid,style={draw,circle}] {$m_2$}
                child[grow=down,solid] { node {$T^{m_2}_0$} }
                child[solid] { node[xshift=15] {$T^{m_2}_{1,\dots, s(m_2)}$} } 
                } 
            } 
        }
        child { node[xshift=10] {$T^f_{1,\dots,s(f)}$} } 
    };
\end{tikzpicture}
}}
\subfigure{
\scalebox{.7}{\begin{tikzpicture}[very thick]
\draw [thick] [->] (2,10) -- (3,10);
\node at (2.5,10) [above] {$(c,d)$};
\draw [white] (2,3) -- (3,3);
\end{tikzpicture}
}}
\subfigure{
\scalebox{.6}{
\begin{tikzpicture}[very thick]
  \node [style={draw,circle}]{d}
    child { node[style={draw,circle}] {e} 
        child { node[xshift=-10]    {$T^e_{0,\dots,s(e)-1}$} }
        child[grow=down]{ {} 
            child[grow=down,dashed]{ node[solid,style={draw,circle}] {g}
                child[solid]{ node[xshift=-15] {$T^g_{0,\dots,s(g)-1}$} }
                child[solid] { node[style={draw,circle},xshift=-20] {k}
                    child[solid] { node[xshift=-10] {$T^b_{0,\dots,s(b)-1}$} }
                    child[grow=down] { {} 
                        child[dashed]{ node[solid,style={draw,circle}] {h}
                            child[solid]{ node[xshift=-20] {$T^h_{0,\dots,s(h)-1}$} }
                            child[solid,grow=down] { node [style={draw,circle}] {a} 
                                child { node {$T^a_0$} }
                                child[grow=down]{ node {$T^a_{1,\dots,s(a)-1}$} }
                                child { node[style={draw,circle},xshift=-10] {} } 
                            }
                        }
                    }    
                } 
            }
        }
    }
    child { node[style={draw,circle},xshift=25] {f} 
        child[grow=down] { {}
            child[grow=down] { {}
            child[dashed,grow=down] { node[solid,style={draw,circle}] {$m_2$}
                child[solid,grow=down] { node[style={draw,circle}] {c}
                    child { node[style={draw,circle},xshift=25] {} }
                    child[grow=down,solid,xshift=-15]{ {}
                        child[grow=down,dashed] { node[solid,style={draw,circle}] {$m_1$}
                            child[grow=down,solid] { node {$T^{m_1}_0$} } 
                            child[solid] { node[xshift=20] {$T^{m_1}_{1,\dots, s(m_1)}$} }  
                            }
                        }    
                    child { node {$T^c_{2,\dots,s(c)-2}$ } }
                    child { node {$T^{m_2}_0$} }
                }
                child[solid] { node[xshift=15] {$T^{m_2}_{1,\dots, s(m_2)}$} } 
                } 
            } 
        }
        child { node[xshift=10] {$T^f_{1,\dots,s(f)}$} } 
    };
\end{tikzpicture}
}}
\subfigure{
\scalebox{.7}{\begin{tikzpicture}[very thick]
\draw [thick] [->] (2,10) -- (3,10);
\node at (2.5,10) [above] {$(a,d)$};
\draw [white] (2,3) -- (3,3);
\end{tikzpicture}
}}
\subfigure{
\scalebox{.6}{
\begin{tikzpicture}[very thick]
  \node [style={draw,circle}]{d}
    child { node[style={draw,circle}] {e} 
        child { node[xshift=-10]    {$T^e_{0,\dots,s(e)-1}$} }
        child[grow=down]{ {} 
            child[grow=down,dashed]{ node[solid,style={draw,circle}] {g}
                child[solid]{ node[xshift=-15] {$T^g_{0,\dots,s(g)-1}$} }
                child[solid] { node[style={draw,circle},xshift=-20] {k}
                    child[solid] { node[xshift=-10] {$T^b_{0,\dots,s(b)-1}$} }
                    child[grow=down] { {} 
                        child[dashed]{ node[solid,style={draw,circle}] {h}
                            child[solid]{ node[xshift=-20] {$T^h_{0,\dots,s(h)-1}$} }
                            child[grow=down,solid] { node {$T^a_0$} }
                        }
                    }    
                } 
            }
        }
    }
    child { node[style={draw,circle},xshift=50] {f} 
        child[grow=down] { {}
            child[grow=down] { {}
            child[dashed,grow=down] { node[solid,style={draw,circle}] {$m_2$}
                child[solid,grow=down] { node[style={draw,circle}] {c}
                    child[solid] { node[style={draw,circle},xshift=20] {a} 
                        child { node[style={draw,circle},xshift=10] {} }
                        child[grow=down]{ node {$T^a_{1,\dots,s(a)-1}$} }
                        child { node[style={draw,circle},xshift=-10] {} } 
                    }
                    child[grow=down,solid,xshift=10]{ {}
                        child[grow=down,dashed] { node[solid,style={draw,circle}] {$m_1$}
                            child[grow=down,solid] { node {$T^{m_1}_0$} } 
                            child[solid] { node[xshift=20] {$T^{m_1}_{1,\dots, s(m_1)}$} }  
                            }
                        }    
                    child { node[xshift=15] {$T^c_{2,\dots,s(c)-2}$ } }
                    child { node[xshift=10] {$T^{m_2}_0$} }
                }
                child[solid] { node[xshift=15] {$T^{m_2}_{1,\dots, s(m_2)}$} } 
                } 
            } 
        }
        child { node[xshift=10] {$T^f_{1,\dots,s(f)}$} } 
    };
\end{tikzpicture}
}}
\subfigure{
\scalebox{.7}{\begin{tikzpicture}[very thick]
\draw [thick] [->] (2,10) -- (3,10);
\node at (2.5,10) [above] {$(a,c)$};
\draw [white] (2,3) -- (3,3);
\end{tikzpicture}
}}
\subfigure{
\scalebox{.6}{
\begin{tikzpicture}[very thick]
  \node [style={draw,circle}]{d}
    child { node[style={draw,circle}] {e} 
        child { node[xshift=-10]    {$T^e_{0,\dots,s(e)-1}$} }
        child[grow=down]{ {} 
            child[grow=down,dashed]{ node[solid,style={draw,circle}] {g} 
                child[solid]{ node[xshift=-15] {$T^g_{0,\dots,s(g)-1}$} }
                child[solid] { node[style={draw,circle},xshift=-20] {k}
                    child[solid] { node[xshift=-10] {$T^b_{0,\dots,s(b)-1}$} }
                    child[grow=down] { {} 
                        child[dashed]{ node[solid,style={draw,circle}] {h}
                            child[solid]{ node[xshift=-20] {$T^h_{0,\dots,s(h)-1}$} }
                            child[grow=down,solid] { node {$T^a_0$} }
                        }
                    }    
                }
            }
        }
    }
    child { node[style={draw,circle},xshift=20] {f} 
        child[grow=down] { {}
            child[dashed,grow=down] { node[solid,style={draw,circle}] {$m_2$}
                child[solid,grow=down] { node[style={draw,circle}] {c} 
                    child { node[style={draw,circle},xshift=20] {} }
                    child[grow=down,xshift=-20] { {}
                        child[dashed,grow=down] { node[solid,style={draw,circle}] {$m_1$}
                            child[solid,grow=down] { node [style={draw,circle}] {a}
                                child { node[style={draw,circle},xshift=5] {} }
                                child[grow=down]{ node {$T^a_{1,\dots,s(a)-1}$} }
                                child[solid]{ node {$T^{m_1}_0$} }
                            }
                            child[solid] { node[xshift=15] {$T^{m_1}_{1,\dots,s(m_1)}$} }
                        }
                    }
                    child { node {$T^c_{2,\dots,s(c)-2}$ }}
                    child { node {$T^{m_2}_0$} }
                    }
                child[solid] { node[xshift=15] {$T^{m_2}_{1,\dots, s(m_2)}$} }
                }
            }
        child { node[xshift=10] {$T^f_{1,\dots,s(f)}$} }
    };
\end{tikzpicture}
}}
\caption{The length three side of an $a\in T^c_{0}$ pentagon from \cref{lem:addedinvsetcharact}. $m_1$ is the smallest element of ${_LT^c_1}$ that is larger than $a$ and $m_2$ is the smallest element of $_LT^d_{j+1}$ that is larger than $c$.}
\label{fig:sc0pentalongforproof}
\end{figure}

It remains to show that there is an {\sdt} $P'$ with $\inv{P'}-\inv{T} = A_T(a,b)\cup A_T(c,d) \cup F_T(a,c)$. We claim there is a saturated chain \[\rot{T}{Q}{(c,d)} \rot{}{P}{(a,d)} \rot{}{P'}{(a,c)} \] and that $\inv{P'}-\inv{T} = A_T(a,b)\cup A_T(c,d) \cup F_T(a,c)$. \cref{fig:sc0pentalongforproof} illustrates this chain and guides the proof. 

We first show $(a,d)$ is a tree ascent of $Q$ and then that $(a,c)$ is a tree ascent of the {\sdt} $P$ resulting from the {\str} of $Q$ along $(a,d)$. We recall that to show that $(a,d)$ is a tree ascent of $Q$, it suffices to show that $a\in {_RT^d_j}$ for some $j<s(d)$ and that if $s(a)>0$, then $T^a_{s(a)}$ is a leaf and similarly for $(a,c)$ in $P$.

We observe that $c\in {_RT^d_j}$ for some $j<s(d)$ since $(c,d)$ is a tree ascent of $T$. Also, $a \in {_RT^c_0}$ since $(a,c)$ is a tree ascent of $T$ with $a \in T^c_0$. Then by \cref{rmk:treepicture}, $a \in {_RQ^d_j}$ since $f=c$ was the only $a<f<d$ with $a\in T^f_k$ and $k<s(f)$. Further, \cref{rmk:treepicture} implies $Q^a=T^a$. If $s(a)>0$, then $T^a_{s(a)}$ is a leaf because $(a,c)$ is a tree ascent of $T$. So $Q^a_{s(a)}$ would be a leaf also. Hence, $(a,d)$ is a tree ascent of $Q$.
   
Next we observe that $Q^c_0$ is a leaf by \cref{rmk:treepicture} and the fact that $0<s(c)$ by supposition. Thus, also by \cref{rmk:treepicture}, $P^c_0 = P^a$. Hence, $a \in {_RP^c_0}$. Again, since $Q^c_0$ is a leaf, $P^a_{s(a)}$ is a leaf by \cref{rmk:treepicture}. Hence, $(a,c)$ is a tree ascent of $P$. Therefore, we have the claimed saturated chain.

Now by \cref{lem:inversionsadded}, $\inv{P'}-\inv{T} = A_T(c,d) \cup A_Q(a,d) \cup A_P(a,c)$. But by \cref{rmk:treepicture}, we have $Q^a = T^a$ and $P^a= T^a\setminus 0$. Hence, $A_Q(a,d) = F_T(a,c)$. Further, since $b=c$, $A_P(a,c)=A_T(a,b)$. Therefore, $\inv{Z\vee Q} - \inv{T} = A_T(a,b)\cup A_T(c,d) \cup F_T(a,c)$.
\end{proof}

In the following lemma, we establish that in the interval $[T,Z\vee Q]$ for any $T\precdot Z,Q$, the only atoms are $Z$ and $Q$. We use this in part of the proof that there are only two maximal chains in such an interval. The proof of this lemma relies on \cref{lem:addedinvsetcharact} and our restrictions on tree ascents from \cref{lem:nolowerascents}. We can visually verify this lemma in the three examples of {\swo} given in \cref{fig:swoexamples}.

\begin{lemma}\label{lem:nootheratoms}
Let $T$ be an {\sdt}. Let $1\leq a<b \leq n$ and $1\leq c<d \leq n$ be such that $(a,b)$ and $(c,d)$ are tree ascents of $T$ with $a<c$. Suppose $\rot{T}{Z}{(a,b)}$ and $\rot{T}{Q}{(c,d)}$, then $Z$ and $Q$ are the only atoms in $[T,Z\vee Q]$.
\end{lemma}

\begin{proof}
First, \cref{thm:covers} implies that atoms of the $[T,Z\vee Q]$ correspond to the tree ascents $(e,f)$ of $T$ such that $(f,e)\in \inv{Z\vee Q} - \inv{T}$. Thus, by \cref{lem:addedinvsetcharact} the atoms of $[T,Z\vee Q]$ correspond to pairs $(f,e)\in A_T(a,b)\cup A_T(c,d) \cup F_T(a,c)$ such that $(e,f)$ is a tree ascent of $T$. By \cref{lem:inversionsadded} and \cref{lem:nolowerascents}, the only pairs $(f,e)\in A_T(a,b) \cup A_T(c,d)$ such that $(e,f)$ are tree ascents of $T$ are $(f,e)=(b,a),(d,c)$. Further, if $F_T(a,c)\neq \emptyset$ and $(f,e)\in F_T(a,c)$, then $b=c$, $f=d$, and $e\in T^a\setminus 0$ by \cref{def:tertaddedinvs}. For all $e\in T^a$, $e\in T^b_k$ with $k<s(b)$ since $(a,b)$ is a tree ascent of $T$. Then since $b<d$, $(e,d)$ such that $e\in T^a$ does not satisfy condition (ii) of \cref{def:treeascent}, and so is not a tree ascent of $T$. Therefore, the only atoms of $[T,Z\vee Q]$ are $Z$ and $Q$. 
\end{proof}

In the next lemma, we consider the case of $\rot{T}{Z}{(a,b)}$ and $\rot{T}{Q}{(c,d)}$ for tree ascents $(a,b)$ and $(c,d)$ of $T$, but when $(c,d)$ is not a tree ascent of $Z$. This is one of the cases from \cref{lem:stopbeingascent}. We construct a saturated chain from $T$ to $Z\vee Q$. This is similar to the construction of the saturated chain in the proof of \cref{lem:addedinvsetcharact}. This new chain is illustrated in \cref{fig:sc-1pentalongforproof} below. As an example, we can construct this chain using the {\sdt} in \cref{fig:sdtexample} and its tree ascents $(2,4)$ and $(4,5)$.

\begin{figure}[H]
\centering
\subfigure{
\scalebox{.6}{
\begin{tikzpicture}[very thick]
  \node [style={draw,circle}]{d}
    child { node[style={draw,circle}] {e} 
        child { node[xshift=-10]    {$T^e_{0,\dots,s(e)-1}$} }
        child[grow=down]{ {} 
            child[grow=down,dashed]{ node[solid,style={draw,circle}] {g} 
                child[solid]{ node[xshift=-15] {$T^g_{0,\dots,s(g)-1}$} }
                child[solid,grow=down] { node[style={draw,circle}] {c} 
                    child { node[xshift=-15] {$T^c_0$} } 
                    child { node[xshift=-20] {$T^c_{1,\dots,s(c)-2}$ }}
                    child[grow=down] { node[style={draw,circle}] {k}
                        child { node[xshift=-10] {$T^b_{0,\dots,s(b)-1}$} }
                        child[grow=down] { {} 
                            child[dashed]{ node[solid,style={draw,circle}] {h}
                                child[solid]{ node {$T^h_0$} }
                                child[solid,grow=down] { node [style={draw,circle}] {a} 
                                    child { node {$T^a_0$} }
                                    child[grow=down]{ node {$T^a_{1,\dots,s(a)-1}$} }
                                    child { node[style={draw,circle},xshift=-10] {} } 
                                }
                            }
                        }    
                    } 
                    child { node[style={draw,circle},xshift=-40] {} }
                } 
            }
        }
    }
    child { node[style={draw,circle}] {f} 
        child[grow=down] { {}
            child[dashed,grow=down] { node[solid,style={draw,circle}] {$m_2$}
                child[solid,grow=down] { node[solid,style={draw,circle}] {i} 
                    child[grow=down]{ {}
                        child[dashed,grow=down] { node[solid,style={draw,circle}] {$m_1$}
                            child[solid,grow=down]{ node {$T^{m_1}_0$} }
                            child[solid] { node[xshift=15] {$T^{m_1}_{1,\dots,s(m_1)}$} } 
                        } 
                    }
                    child { node[xshift=10] {$T^i_{1,\dots,s(i)}$} } 
                } 
            child[solid] { node[xshift=15] {$T^{m_2}_{1,\dots, s(m_2)}$} } 
            } 
        }
        child { node[xshift=10] {$T^f_{1,\dots,s(f)}$} } 
    };
\end{tikzpicture}
}}
\subfigure{
\scalebox{.8}{\begin{tikzpicture}[very thick]
\draw [thick] [->] (2,10) -- (3,10);
\node at (2.5,10) [above] {$(a,c)$};
\draw [white] (2,3) -- (3,3);
\end{tikzpicture}
}}
\subfigure{
\scalebox{.6}{
\begin{tikzpicture}[very thick]
  \node [style={draw,circle}]{d}
    child { node[style={draw,circle},xshift=-20] {e} 
        child { node[xshift=-10]    {$T^e_{0,\dots,s(e)-1}$} }
        child[grow=down]{ {} 
            child[grow=down,dashed]{ node[solid,style={draw,circle}] {g} 
                child[solid]{ node[xshift=-15] {$T^g_{0,\dots,s(g)-1}$} }
                child[solid,grow=down] { node[style={draw,circle}] {c} 
                    child { node[xshift=-25] {$T^c_0$} } 
                    child { node[xshift=-30] {$T^c_{1,\dots,s(c)-2}$ }}
                    child[grow=down] { node[style={draw,circle},xshift=-15] {k}
                        child { node[xshift=-10] {$T^b_{0,\dots,s(b)-1}$} }
                        child[grow=down] { {} 
                            child[dashed]{ node[solid,style={draw,circle}] {h}
                                child[solid]{ node {$T^h_0$} }
                                child[solid,grow=down] { node {$T^a_0$} }
                            }
                        }    
                    } 
                    child[solid,grow=down] { node [style={draw,circle},xshift=30] {a} 
                        child { node[xshift=8,style={draw,circle}] {} }
                        child[grow=down]{ node {$T^a_{1,\dots,s(a)-1}$} }
                        child { node[style={draw,circle},xshift=-10] {} } 
                        }
                } 
            }
        }
    }
    child { node[style={draw,circle},xshift=30] {f} 
        child[grow=down] { {}
            child[dashed,grow=down] { node[solid,style={draw,circle}] {$m_2$}
                child[solid,grow=down] { node[solid,style={draw,circle}] {i} 
                    child[grow=down]{ {}
                        child[dashed,grow=down] { node[solid,style={draw,circle}] {$m_1$}
                            child[solid,grow=down]{ node {$T^{m_1}_0$} }
                            child[solid] { node[xshift=15] {$T^{m_1}_{1,\dots,s(m_1)}$} } 
                        } 
                    }
                    child { node[xshift=10] {$T^i_{1,\dots,s(i)}$} } 
                } 
            child[solid] { node[xshift=15] {$T^{m_2}_{1,\dots, s(m_2)}$} } 
            } 
        }
        child { node[xshift=10] {$T^f_{1,\dots,s(f)}$} } 
    };
\end{tikzpicture}
}}
\subfigure{
\scalebox{.8}{\begin{tikzpicture}[very thick]
\draw [thick] [->] (2,10) -- (3,10);
\node at (2.5,10) [above] {$(a,d)$};
\draw [white] (2,3) -- (3,3);
\end{tikzpicture}
}}
\subfigure{
\scalebox{.6}{
\begin{tikzpicture}[very thick]
  \node [style={draw,circle}]{d}
    child { node[style={draw,circle},xshift=-20] {e} 
        child { node[xshift=-10]    {$T^e_{0,\dots,s(e)-1}$} }
        child[grow=down]{ {} 
            child[grow=down,dashed]{ node[solid,style={draw,circle}] {g} 
                child[solid]{ node[xshift=-15] {$T^g_{0,\dots,s(g)-1}$} }
                child[solid,grow=down] { node[style={draw,circle}] {c} 
                    child { node[xshift=-25] {$T^c_0$} } 
                    child { node[xshift=-20] {$T^c_{1,\dots,s(c)-2}$ }}
                    child[grow=down] { node[style={draw,circle},] {k}
                        child { node[xshift=-10] {$T^b_{0,\dots,s(b)-1}$} }
                        child[grow=down] { {} 
                            child[dashed]{ node[solid,style={draw,circle}] {h}
                                child[solid]{ node {$T^h_0$} }
                                child[solid,grow=down] { node {$T^a_0$} }
                            }
                        }    
                    }
                    child { node[style={draw,circle},xshift=-30] {} }
                } 
            }
        }
    }
    child { node[style={draw,circle},xshift=10] {f} 
        child[grow=down] { {}
            child[dashed,grow=down] { node[solid,style={draw,circle}] {$m_2$}
                child[solid,grow=down] { node[solid,style={draw,circle}] {i} 
                    child[grow=down]{ {}
                        child[dashed,grow=down] { node[solid,style={draw,circle}] {$m_1$}
                            child[solid,grow=down] { node [style={draw,circle}] {a} 
                                child { node[xshift=8,style={draw,circle}] {} }
                                child[grow=down]{ node {$T^a_{1,\dots,s(a)-1}$} }
                                child[solid,grow=down]{ node[xshift=35] {$T^{m_1}_0$} } 
                            }
                            child[solid] { node[xshift=15] {$T^{m_1}_{1,\dots,s(m_1)}$} } 
                        } 
                    }
                    child { node[xshift=10] {$T^i_{1,\dots,s(i)}$} } 
                } 
            child[solid] { node[xshift=15] {$T^{m_2}_{1,\dots, s(m_2)}$} } 
            } 
        }
        child { node[xshift=10] {$T^f_{1,\dots,s(f)}$} } 
    };
\end{tikzpicture}
}}
\subfigure{
\scalebox{.8}{\begin{tikzpicture}[very thick]
\draw [thick] [->] (2,10) -- (3,10);
\node at (2.5,10) [above] {$(c,d)$};
\draw [white] (2,3) -- (3,3);
\end{tikzpicture}
}}
\subfigure{
\scalebox{.6}{
\begin{tikzpicture}[very thick]
  \node [style={draw,circle}]{d}
    child { node[style={draw,circle},xshift=-20] {e} 
        child { node[xshift=-10]    {$T^e_{0,\dots,s(e)-1}$} }
        child[grow=down]{ {} 
            child[grow=down,dashed]{ node[solid,style={draw,circle}] {g} 
                child[solid]{ node[xshift=-15] {$T^g_{0,\dots,s(g)-1}$} }
                child[solid,grow=down] { node {$T^c_0$} }
            }
        }
    }
    child { node[style={draw,circle}] {f} 
        child[grow=down] { {}
            child[dashed,grow=down] { node[solid,style={draw,circle}] {$m_2$}
                child[solid,grow=down] { node[style={draw,circle}] {c} 
                    child { node[xshift=-25] {$T^c_0$} } 
                    child { node[xshift=-20] {$T^c_{1,\dots,s(c)-2}$ }}
                    child[grow=down] { node[style={draw,circle},] {k}
                        child { node[xshift=-10] {$T^b_{0,\dots,s(b)-1}$} }
                        child[grow=down] { {} 
                            child[dashed]{ node[solid,style={draw,circle}] {h}
                                child[solid]{ node {$T^h_0$} }
                                child[solid,grow=down] { node {$T^a_0$} }
                            }
                        }    
                    }
                    child[solid] { node[solid,style={draw,circle},xshift=-30] {i} 
                        child[grow=down]{ {}
                            child[dashed,grow=down] { node[solid,style={draw,circle}] {$m_1$}
                                child[solid,grow=down] { node [style={draw,circle}] {a} 
                                    child { node[xshift=8,style={draw,circle}] {} }
                                    child[grow=down]{ node {$T^a_{1,\dots,s(a)-1}$} }
                                    child[solid,grow=down]{ node[xshift=35] {$T^{m_1}_0$} } 
                                }
                                child[solid] { node[xshift=15] {$T^{m_1}_{1,\dots,s(m_1)}$} } 
                            } 
                        }
                        child { node[xshift=10] {$T^i_{1,\dots,s(i)}$} } 
                    }
                }
            child[solid] { node[xshift=15] {$T^{m_2}_{1,\dots, s(m_2)}$} } 
            } 
        }
        child { node[xshift=10] {$T^f_{1,\dots,s(f)}$} } 
    };
\end{tikzpicture}
}}
\caption{The length three side of an $a\in T^c_{s(c)-1}$ pentagon from \cref{lem:chaininsc-1pent}. $m_1$ is the smallest element of ${_LT^d_{j+1}}$ that is larger than $a$ and $m_2$ is the smallest element of $_LT^d_{j+1}$ that is larger than $c$.}
\label{fig:sc-1pentalongforproof}
\end{figure}
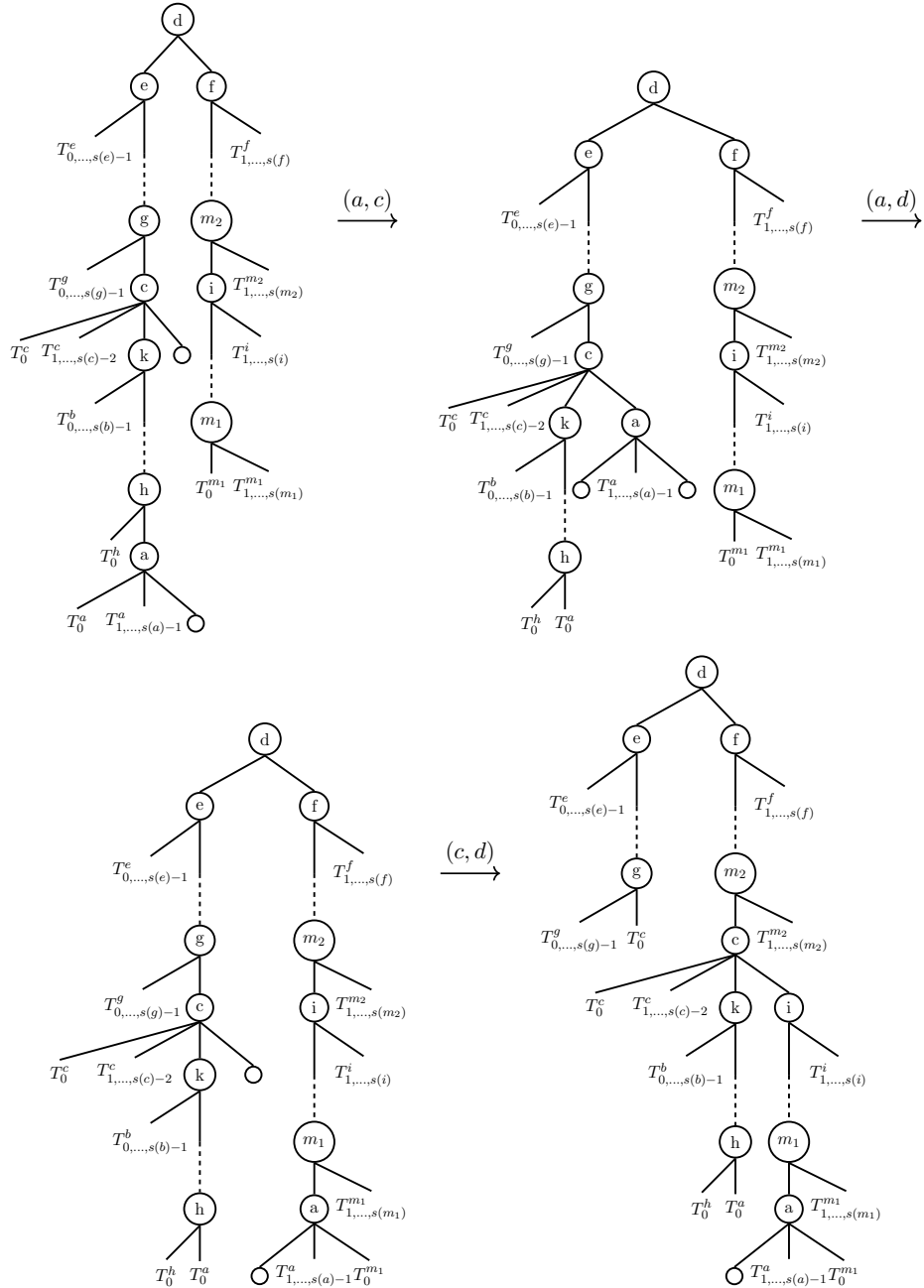

\begin{lemma}\label{lem:chaininsc-1pent}
Let $T\precdot Z,Q$ be cover relations in {\swo} corresponding to $\rot{T}{Z}{(a,b)}$ and $\rot{T}{Q}{(c,d)}$ for tree ascents $(a,b)$ and $(c,d)$ of $T$ with $a<c$. Suppose $(c,d)$ is not a tree ascent of $Z$, then there is a saturated chain of the form $\rot{T}{Z}{(a,b)}\rot{}{P}{(a,d)}\rot{}{Z\vee Q}{(c,d)}.$
\end{lemma}

\begin{proof}
First, by \cref{lem:stopbeingascent}, $b=c$, $a\in T^c_{s(c)-1}$, and $s(c)>0$. So the two tree ascents of interest in $T$ are $(a,c)$ and $(c,d)$. We claim that there is a saturated chain \[\rot{T}{Z}{(a,c)}\rot{}{P}{(a,d)} \rot{}{P'}{(c,d)}. \] We first show that $(a,d)$ is a tree ascent of $Z$, and then that $(c,d)$ is a tree ascent of the {\sdt} $P$ resulting from the {\str} of $Z$ along $(a,d)$. Then we show that $P'=Z\vee Q$. This is illustrated in \cref{fig:sc-1pentalongforproof} which also guides the proof.

First, we note that $c\in {_RT^d_j}$ for some $0\leq j<s(d)$ since $(c,d)$ is a tree ascent of $T$. Thus, $c\in {_RZ^d_j}$ because the tree rotation of $T$ along $(a,b)$ does not move any vertices above $a$ in $T$. Also, $T^c_{s(c)}$ is a leaf because $(c,d)$ is a tree ascent of $T$ and $s(c)>0$. Then by \cref{rmk:treepicture}, $Z^c_{s(c)}=T^a\setminus 0$ so $a$ is the $s(c)$th child $c$ in $Z$. Thus, $a\in {_RZ^d_j}$ since $a$ is the $s(c)$th child of $c$ in $Z$. Further, $Z^a_{s(a)}$ is a leaf again by \cref{rmk:treepicture} and the fact that $T^c_{s(c)}$ is a leaf. Hence, $(a,d)$ is a tree ascent of $Z$. 

Now, again by \cref{rmk:treepicture}, $c\in  {_RP^d_j}$ where $j$ is the same $j$ as above so $0\leq j<s(d)$. Lastly, $P^c_{s(c)} = Z^a_0$ which is a leaf by \cref{rmk:treepicture} and the fact that $T^c_{s(c)}$ is a leaf. Thus, $(c,d)$ is a tree ascent of $P$.
    
Now we claim $P'=Z\vee Q$. By \cref{lem:inversionsadded}, $\inv{P'}-\inv{T}= A_T(a,c) \cup A_Z(a,d) \cup A_P(c,d)$. Thus, by \cref{lem:addedinvsetcharact}, it remains to show that $A_T(a,c) \cup A_Z(a,d) \cup A_P(c,d) = A_T(a,b)\cup A_T(c,d) \cup F_T(a,c)$. Since $b=c$, $A_T(a,b)=A_T(a,c)$. To show $A_Z(a,d)\cup A_P(c,d) = A_T(c,d) \cup F_T(a,c)$, there are two cases because $b=c$. Either $a\in T^c_0$ or $a\not \in T^c_0$, that is, $F_T(a,c)$ is possibly non-empty or $F_T(a,c)=\emptyset$, respectively, by \cref{def:tertaddedinvs}.

Suppose $a\in T^c_0$. Then, as above, by \cref{rmk:treepicture} and the fact that $T^c_{s(c)}$ is a leaf, $Z^a=T^a\setminus 0$. Thus, by \cref{lem:inversionsadded} and \cref{def:tertaddedinvs}, $A_z(a,d)=F_T(a,c)$. Further, by \cref{rmk:treepicture} along with the fact that $a\in T^c_0$ and our previous observations that $P^c_{s(c)}$ and $T^c_{s(c)}$ are leaves, $P^c\setminus 0 = T^c\setminus 0$. Thus, by \cref{lem:inversionsadded}, $A_P(c,d)=A_T(c,d)$. 

Now suppose $a\not \in T^c_0$ so $F_T(a,c)=\emptyset$. To show that $A_Z(a,d)\cup A_P(c,d) = A_T(c,d)$, we need to show that $T^c\setminus 0 = Z^a\setminus 0 \cup P^c\setminus 0$ as sets of labeled vertices. We previously argued that $Z^c_s(c) =Z^a\setminus 0 = T^a\setminus 0$. Also, as sets of labeled vertices $P^c\setminus 0 = \left(T^c\setminus 0\right) \setminus \left(T^a\setminus 0\right)$ by \cref{rmk:treepicture}. This completes the proof.
\end{proof}

In the next three lemmas, we begin with $[T,Z\vee Q]$ for $T\precdot Z,Q$ having $(a,b)$ and $(c,d)$ the tree ascents of $T$ associated with $Z$ and $Q$, respectively. We prove that three of the relationships given by \cref{lem:stopbeingascent} result in $[T,Z\vee Q]$ having Hasse diagram that is a diamond or a pentagon and that, in any of these three cases, our labeling in \cref{thm:sblabelingsection} satisfies the conditions of an SB-labeling. \cref{thm:covers} characterizing cover relations in {\swo} and \cref{lem:joinoftwoats} along with the chains constructed in \cref{lem:addedinvsetcharact} and \cref{lem:chaininsc-1pent}, establish the two maximal chains of $[T,Z\vee Q]$ in these cases. Thus, the bulk of the proofs the next three lemmas is showing that there are no other maximal chains in $[T,Z\vee Q]$ in these cases using \cref{lem:nootheratoms} and \cref{lem:nolowerascents}. Moreover, the proofs for the two distinct ways a pentagonal interval can arise combine to prove this about the hexagonal case in our proof of our main result \cref{thm:sblabeling}. We note our labeling always satisfies the first condition of an SB-labeling by \cref{rmk:nosamebottomascents}. We also not that all three lemmas can be verified on the appropriate intervals of the examples of {\swo} in (b) and (c) of \cref{fig:swoexamples}. 

\begin{lemma}\label{lem:diamondints}
Let $T\precdot Z,Q$ be cover relations in {\swo} corresponding to $\rot{T}{Z}{(a,b)}$ and $\rot{T}{Q}{(c,d)}$ for tree ascents $(a,b)$ and $(c,d)$ of $T$ with $a<c$. Suppose $(a,b)$ is a tree ascent of $Q$ and $(c,d)$ is a tree ascent of $Z$. Then $[T, Z\vee Q]$ has Hasse diagram which is a diamond and the edge labeling of \cref{thm:sblabelingsection} on its two maximal chains satisfies \cref{def:sblabeling}.
\end{lemma}

\begin{proof}
By \cref{lem:joinoftwoats}, $\inv{Z\vee Q} = \tcp{\inv{Z}}{(d,c)}= \tcp{\inv{Q}}{(b,a)}$. Then since $(c,d)$ is a tree ascent of $Z$ and $(a,b)$ is a tree ascent of $Q$, $\rot{R}{Z\vee Q}{(c,d)}$ and $\rot{Q}{Z\vee Q}{(a,b)}$. Hence, $R,Q\precdot Z\vee Q$ Thus, $T\precdot Z \precdot Z\vee Q$ and $T \precdot Q \precdot Z\vee Q$ are two distinct saturated chains from $T$ to $Z\vee Q$. Then to show there is not a third such saturated chain, it suffices to show there is not a third atom in the interval $[T, Z\vee Q]$. We showed this fact as \cref{lem:nootheratoms}.
        
Now we observe that the label sequences of the saturated chains $T\precdot Z \precdot Z\vee Q$ and $T \precdot Q \precdot Z\vee Q$ are $a,c$ and $c,a$, respectively. Therefore, \cref{def:sblabeling} is satisfied.
\end{proof}

\begin{lemma}\label{lem:pentaintssc-1}
Let $T\precdot Z,Q$ be cover relations in {\swo}  corresponding to $\rot{T}{Z}{(a,b)}$ and $\rot{T}{Q}{(c,d)}$ for tree ascents $(a,b)$ and $(c,d)$ of $T$ with $a<c$. Suppose $(a,b)$ is a tree ascent of $Q$ and $(c,d)$ is not a tree ascent of $Z$. Then $[T, Z\vee Q]$ has Hasse diagram which is a pentagon and the edge labeling of \cref{thm:sblabelingsection} on its two maximal chains satisfies \cref{def:sblabeling}.
\end{lemma}

\begin{proof}
\cref{fig:sc-1pentalongforproof} illustrates this case and provides a guide for this proof. First, we observe that $Q\precdot Z\vee Q$ by \cref{lem:joinoftwoats} because $(a,b)$ is a tree ascent of $Q$. This cover relation is given by the {\str} $\rot{Q}{Z\vee Q}{(a.b)}$. Thus, the label sequence for the saturated chain $T\precdot Q \precdot Z\vee Q$ is $c,a$. 
    
Next, by \cref{lem:stopbeingascent}, $b=c$ and $a\in T^c_{s(c)-1}$ with $s(c)-1>0$. Then by \cref{lem:chaininsc-1pent} there is a saturated chain of the form $\rot{T}{Z}{(a,c)}\rot{}{P}{(a,d)}\rot{}{Z\vee Q}{(c,d)}.$
    
Thus, it remains to show that there are no other maximal chains in $[T,Z\vee Q]$ in this case. \cref{lem:inversionsadded} shows $Q\not \preceq P$. Thus, it suffices to show there are no other elements in $[T,Z\vee Q]$ besides $T,Z,Q,P,Z\vee Q$.
    
We note the only atoms in $[T,Z\vee Q]$ are $Z$ and $Q$ by \cref{lem:nootheratoms}. Then since $Q\precdot Z\vee Q$, the only other possibility of an element in $[T,Z\vee Q]$ besides the five listed above is that there is an atom of $[Z,Z\vee Q]$ distinct from $P$. Assume there is such an atom, $Z'$. Then by \cref{thm:covers} and  \cref{lem:inversionsadded}, there exists $(f,e)\in A_Z(a,d) \cup A_P(c,d)$ such that $(e,f)$ is a tree ascent of $Z$ with $\rot{Z}{Z'}{(e,f)}$. Now by \cref{lem:inversionsadded} and \cref{lem:nolowerascents}, the only pair $(f,e)\in A_Z(a,d)$ such that $(e,f)$ is a tree ascent of $Z$ is $(f,e)=(d,a)$. However, $(f,e)\neq (d,a)$ since $Z'\neq P$. Next we note that any $(f,e)\in A_P(c,d)$ has the form $(d,e)$ for some $e\in P^c\setminus 0$ by \cref{lem:inversionsadded}. We observe that by \cref{rmk:treepicture}, $P^c=Z^c\setminus s(c)$. Thus, any such any $e\in P^c\setminus 0$ with $e\neq c$ has $e\in Z^c_i$ with $i\neq s(c)$. Thus, $(e,d)$ does not satisfy (ii) of \cref{def:treeascent} of $(e,d)$ being a tree ascent of $Z$ because $e<c<d$. Thus, $(c,d)$ must be the tree ascent of $Z$ corresponding to $Z'$. However, this contradicts the hypothesis of the lemma that $(c,d)$ is not a tree ascent of $Z$. Hence, $P$ is the only atom of $[Z,Z\vee Q]$, and there are no other elements of $[T,Z\vee Q]$ besides the five listed earlier.
    
The two saturated chains have label sequences $c,a$ and $a,a,c$ which satisfy \cref{def:sblabeling}.
\end{proof}

\begin{lemma}\label{lem:pentaints0}
Let $T\precdot Z,Q$ be cover relations in {\swo} corresponding to $\rot{T}{Z}{(a,b)}$ and $\rot{T}{Q}{(c,d)}$ for tree ascents $(a,b)$ and $(c,d)$ of $T$ with $a<c$. Suppose $(a,b)$ is not a tree ascent of $Q$, but $(c,d)$ is a tree ascent of $Z$. Then $[T, Z\vee Q]$ has Hasse diagram which is a pentagon and the edge labeling of \cref{thm:sblabelingsection} on its two maximal chains satisfies \cref{def:sblabeling}.
\end{lemma}

\begin{proof}
In this case, $Z\precdot Z\vee Q$ by \cref{lem:joinoftwoats}. This cover relation is given by the {\str} $\rot{Z}{Z\vee Q}{(c,d)}$. Thus, there is a saturated chain $T\precdot R\precdot Z\vee Q$ with label sequence $a,c$.
    
Since $(a,b)$ is not a tree ascent of $Q$, $b=c$ and $a \in T^c_0$ with $s(c)>1$ by \cref{lem:stopbeingascent}. Then by the proof of \cref{lem:addedinvsetcharact}, there is a saturated chain of the form \[\rot{T}{Q}{(c,d)}\rot{}{P}{(a,d)}\rot{}{Z\vee Q}{(a,c)}. \]

Thus, it remains to show these are the only saturated chains in the interval $[T,Z\vee Q]$. Again \cref{lem:inversionsadded} implies $Z\not \preceq P$. Hence, it suffices to show there are no other elements in $[T,Z\vee Q]$ besides $T,Z,Q,P,Z\vee Q$.
    
Again the only atoms in $[T,Z\vee Q]$ are $Z$ and $Q$ by \cref{lem:nootheratoms}. Since $Z\precdot Z\vee Q$, the only other possibility is that there is an atom $Q'$ in $[Q,Z\vee Q]$ distinct from $P$. Assume $Q'$ is such an atom. Then by \cref{thm:covers} and \cref{lem:inversionsadded}, there exists $(f,e) \in A_Q(a,d) \cup A_P(a,c)$ such that $(e,f)$ is a tree ascent of $Q$ and $\rot{Q}{Q'}{(e,f)}$. By \cref{lem:nolowerascents} and \cref{lem:inversionsadded}, the only pair $(f,e)\in A_Q(a,d)$ such that $(e,f)$ is a tree ascent of $Q$ is $(f,e)=(d,a)$. But $(f,e)\neq (d,a)$ since $Q'\neq P$. Next we note that any $(f,e)\in A_P(a,c)$, has the form $(c,e)$ for some $e\in P^a\setminus 0$. By \cref{rmk:treepicture}, $P^a\setminus 0 = Q^a\setminus 0 = T^a\setminus 0$ since $a\in T^c_0$ and $s(c)>1$. Also by \cref{rmk:treepicture}, no element of $T^a$ is in $Q^c$ since $a\in T^c_0$. Thus, for $e \in P^a\setminus 0$, $e\not \in Q^c$. Thus, no $(f,e)\in A_P(a,c)$ has $(e,f)$ a tree ascent of $Q$. Hence, $P$ is the only atom of $[Q,Z\vee Q]$.
    
Lastly, the label sequences for these two chains are $a,c$ and $c,a,a$ which satisfy \cref{def:sblabeling}.
\end{proof}

This brings us to the proof of our main theorem, namely that \cref{thm:sblabelingsection} is an SB-labeling of $s$-weak order. In the proof, we must consider the four cases for relationships between two tree ascents of an {\sdt} given in \cref{lem:stopbeingascent}. The result in the first three cases was proven in \cref{lem:diamondints}, \cref{lem:pentaintssc-1}, and \cref{lem:pentaints0}. The proof for the fourth case comes from combining the proofs of \cref{lem:pentaintssc-1} and \cref{lem:pentaints0}.

\begin{theorem}\label{thm:sblabeling}
Let $T\precdot Z$ be a cover relation in {\swo}. Let $\rot{T}{Z}{(a,b)}$ be the {\str} of $T$ along the unique tree ascent $(a,b)$ associated to $T\precdot Z$ by \cref{thm:covers}. Let $\lambda$ to be the edge labeling $\lambda(T, Z)= a$. Then $\lambda$ is an SB-labeling of {\swo}.
\end{theorem}

\begin{proof}
Suppose $T\precdot Z,Q$ correspond to $\rot{T}{Z}{(a,b)}$ and $\rot{T}{Q}{(c,d)}$ for tree ascents of $(a,b)$ and $(c,d)$ of $T$ with $a<c$. By \cref{rmk:nosamebottomascents}, $\lambda$ satisfies property (i) of \cref{def:sblabeling}. To verify properties (ii) and (iii) of \cref{def:sblabeling}, there are four cases we must check:
\begin{enumerate}
    \item [(1)] $(a,b)$ is a tree ascent of $Q$ and $(c,d)$ is a tree ascent of $Z$, or
    \item [(2)] $(a,b)$ is a tree ascent of $Q$ while $(c,d)$ is not a tree ascent of $Z$, or
    \item [(3)] $(c,d)$ is a tree ascent of $Z$ while $(a,b)$ is not a tree acsent of $Q$, or
    \item [(4)] $(a,b)$ is not a tree ascent of $Q$ and $(c,d)$ is not a tree ascent of $Z$.
\end{enumerate}

Case (1) is \cref{lem:diamondints}. Case (2) is \cref{lem:pentaintssc-1}. Case (3) is \cref{lem:pentaints0}. Case (4) results in $[T,Z\vee Q]$ having Hasse diagram which is a hexagon and follows from \cref{lem:pentaintssc-1} and \cref{lem:pentaints0} and their proofs as we show now. 

In case (4), \cref{lem:stopbeingascent} implies $b=c$, but this time $a\in T^c_0$ and $s(c)=1$ so $a\in T^c_{s(c)-1}$. Then the proofs of \cref{lem:pentaintssc-1} and \cref{lem:pentaints0} imply that there are two distinct maximal chains in $[T,Z\vee Q]$. Both maximal chains are of length three and their label sequences are $a,a,c$ and $c,a,a$. Additionally, the proofs that there are no other maximal chains in the intervals in \cref{lem:pentaintssc-1} and \cref{lem:pentaints0} combine to show there are no other maximal chains in $[T,Z\vee Q]$. Thus, (ii) and (iii) of \cref{def:sblabeling} are satisfied. Therefore, $\lambda$ is an SB-labeling of {\swo}.
\end{proof}

Thus, we can characterize the homotopy types of open intervals in {\swo} and the M{\"o}bius function of {\swo} as follows.

\begin{corollary}\label{cor:homotopandmobius}
Let $T\preceq Z$ in {\swo}. Then $\Delta(T,Z)$, the order complex of the open interval $(T,Z)$, is homotopy equivalent to a ball or a sphere of some dimension. Moreover, the M{\"o}bius function of {\swo} satisfies $\mu(T,Z) \in \set{-1,0,1}$.
\end{corollary}

\begin{proof}
The characterization of homotopy type follows from \cref{thm:sbthm} and \cref{thm:sblabeling}. The result on the M{\"o}bius function follows from the fact that $\mu(T,Z) = \tilde{\chi} (\Delta(T,Z))$ along with the fact that the reduced Euler characteristic of a ball is 0 and a $d$-sphere is $(-1)^d$. 
\end{proof}

Lastly, we give an intrinsic characterization of the intervals which are homotopy spheres and the dimension of those spheres.

\begin{lemma}\label{lem:lesssharpcharacspheres}
If $T \prec Z$ in {\swo}, then $Z$ is the join of the atoms in $[T,Z]$ if and only if \[ \inv{Z} = \left( \inv{T} + A_T(a_1,b_1) + \dots + A_T(a_l,b_l) \right)^{tc} \] where $(a_1,b_1), \dots, (a_l,b_l)$ are the tree ascents of $T$ such that $(b_i,a_i)\in \inv{Z}-\inv{T}$. Moreover, the number of atoms in the interval $[T,Z]$ is $l$ regardless of whether or not $Z$ is the join of atoms in the interval.
\end{lemma}

\begin{proof}
Let $T\preceq Z$ in {\swo}. The number of atoms in $[T,Z]$ follows from the characterization of cover relations in {\swo}.

Let $(a_1,b_1), \dots, (a_l,b_l)$ be all the tree ascents of $T$ contained in $\inv{Z}-\inv{T}$. Let $T_1,\dots,T_l$ be the corresponding atoms of $[T,Z]$, respectively. Then to prove the characterization of the join of atoms, it suffices to show $\inv{\bigvee_{i=1}^l T_i} = \left( \inv{T} + A_T(a_1,b_1) + \dots + A_T(a_l,b_l) \right)^{tc}$. We note that by induction, $\inv{\bigvee_{i=1}^l T_i} = \left(\inv{T_1}\cup \dots \cup \inv{T_l} \right)^{tc}$. Now by \cref{lem:inversionsadded}, $\inv{T_i} = \inv{T} + A_T(a_i,b_i)$. By \cref{lem:nocommoninvsadded}, the sets $A_T(a_i,b_i)$ are pairwise disjoint. Thus, $$\inv{T} + A_T(a_1,b_1) + \dots + A_T(a_l,b_l) \subset \inv{T_1}\cup \dots \cup \inv{T_l}$$ so $$\left( \inv{T} + A_T(a_1,b_1) + \dots + A_T(a_l,b_l) \right)^{tc} \subset \inv{\bigvee_{i=1}^l T_i}.$$ On the other hand, $\inv{T} + A_T(a_i,b_i) \subset  \inv{T} + A_T(a_1,b_1) + \dots + A_T(a_l,b_l)$ for each $i\in [l]$ so $\inv{T_i} \subset \left( \inv{T} + A_T(a_1,b_1) + \dots + A_T(a_l,b_l) \right)^{tc}$ for each $i\in [l]$. Thus, $\inv{\bigvee_{i=1}^l T_i} \subset \left( \inv{T} + A_T(a_1,b_1) + \dots + A_T(a_l,b_l) \right)^{tc}$ which gives the result.
\end{proof}

\cref{lem:lesssharpcharacspheres} combined with \cref{thm:sbthm} implies the following intrinsic description of intervals which are homotopy spheres and the dimensions of those spheres.

\begin{theorem}\label{thm:countthespheres}
If $T \prec Z$, then $\Delta (T,Z)$ is homotopy equivalent to a sphere if and only if \[\inv{Z} = \left(\inv{T} + A_T(a_1,b_1) + \dots + A_T(a_l,b_l) \right)^{tc}\] where $(a_1,b_1), \dots, (a_l,b_l)$ are the tree ascents of $T$ such that $(b_i,a_i)\in \inv{Z}-\inv{T}$. Moreover, in this case the dimension of the sphere is $l-2$.
\end{theorem}
\end{section}

\begin{section}{An SB-labeling of the $s$-Tamari lattice}\label{sec:stam}

In this section, we prove that a quite similar edge labeling of the $s$-Tamari Lattice is an SB-labeling. The notation and notions we need to work with the {\staml} are defined in \cref{sec:bkgdstam} and are quite similar to those for {\swo}. We use a subscript of $Tam$ to differentiate between {\swo} and the {\staml}, for instance $\precdot_{Tam}$ instead of $\precdot$ for cover relations. For the join however, we still use $\vee$ as in {\swo} because the {\staml} is a sublattice of {\swo}. We follow a quite similar structure of lemmas as in the proof for {\swo}. The proofs are quite similar to the case of {\swo} with the only major difference being that $[T,Z\vee Q]_{Tam}$ for any $T\precdot_{Tam}Z,Q$ have Hasse diagrams which are only diamonds or pentagons. Further, there is only one way that pentagonal intervals arise. There are also some minor differences in the details we must check, but these details are usually simpler than in the case of {\swo} because Tamari tree ascents are always a pair of a parent and child as defined just after \cref{thm:stamarilattice}. Because of the similarities, the proofs presented here are more cursory.

Intuitively, we label cover relations in the {\staml} by the label of the root vertex of the subtree that is moved to obtain the cover relation, that is we label by the smaller element of the Tamari tree ascent associated to the cover relation by \cref{thm:stamaricovers}, just as in {\swo}.

\begin{definition}\label{def:sblabelingtamari}
Let $T\precdot_{Tam} Z$ be a cover relation in the {\staml}. Let $\trot{T}{Z}{(a,b)}$ be the $s$-Tamari rotation of $T$ along the Tamari tree ascent $(a,b)$ of $T$ associated to $T\precdot_{Tam} Z$ by \cref{thm:stamaricovers}. Define $\lambda$ be the edge labeling $\lambda(T, Z)= a$.
\end{definition}

For $T\precdot_{Tam}Z,Q$, we prove that $[T,Z\vee Q]_{Tam}$ has Hasse diagram which is either a diamond or a pentagon, and that the labeling on the two maximal chains satisfies \cref{def:sblabeling} in either case. In the $s$-Tamari lattice, there is only one type of pentagonal interval instead of two. Similarly to {\swo}, our first proposition restricts the Tamari tree ascents which can occur in an {\stt}. We use it to characterize when $[T,z\vee Q]_{Tam}$ has Hasse diagram which is a diamond or which is a pentagon, as well as to describe the atoms in such intervals.
 
\begin{proposition}\label{lem:nolowertamascents}
Let $T$ be an {\stt} and let $1\leq a<b \leq n$ be such that $(a,b)$ is a Tamari tree ascent of $T$. Then no pair of the form $(c,b)$ such that $c\in T^a$ and $c<a$ is a Tamari tree ascent of $T$. 
\end{proposition}

\begin{proof}
Since $(a,b)$ is a Tamari tree ascent of $T$, $a$ is a child of $b$ in $T$. No other $c\in T^a$ is a child of $b$ in $T$.
\end{proof}

Just as in the {\swo} case, the next two definitions let us describe $\inv{Z\vee Q}$ when $T\precdot{Tam}Z,Q$. The subsequent proposition explicitly computes the tree inversions added by an s-Tamari rotation along a Tamari tree ascent.

\begin{definition}\label{def:taminvsaddedset}
Let $T$ be a {\stt} and let $1\leq a<b \leq n$ be such that $(a,b)$ is a Tamari tree ascent of $T$. Let $Z$ be the {\stt} obtained by $\trot{T}{Z}{(a,b)}$. Define \textbf{the set of inversions added by the $s$-Tamari rotation along $(a,b)$}, denoted $\boldsymbol{A^{Tam}_T(a,b)}$, by \[A^{Tam}_T(a,b) =\sett{(f,e)}{\card{Z}{f}{e}>\card{T}{f}{e}}. \]
\end{definition}

\begin{definition}\label{def:terttaminvsadded}
Let $T$ be an {\stt} and let $(a,b)$ and $(c,d)$ be Tamari tree ascents of $T$ with $a<c$. We note that $b$ and $d$ are determined by $a$ and $c$ since they are the parents of $a$ and $c$, respectively. Define the following set valued function: \[F^{Tam}_T(a,c) =\begin{cases} \sett{(d,e)}{e\in T^a\setminus 0} & \text{ if } b=c \text{ and }a\in T^c_0 \\
\\
\emptyset & \text{otherwise}
\end{cases} \]
\end{definition}

\begin{proposition}\label{lem:taminversionsadded}
Let $T$ be an {\stt} and let $1\leq a<b \leq n$ be such that $(a,b)$ is a Tamari tree ascent of $T$. Suppose $\trot{T}{Z}{(a,b)}$. Then $(f,e) \in A^{Tam}_T(a,b)$ if and only if $f=b$ and $e\in T^a\setminus 0$ in which case $$\#_Z(f,e)=\#_T(f,e) + 1.$$
\end{proposition}

\begin{proof}
This follows from \cref{rmk:stamaritreepic} by keeping track of the only subtrees that change in an $s$-Tamari rotation.
\end{proof}

Again as in the {\swo} case, we use the following lemma in one of two different characterizations of $Z\vee Q$ for $T\precdot_{Tam}Z,Q$.   

\begin{lemma}\label{lem:nocommontaminvsadded}
Let $T$ be an {\stt}. Let $1\leq a<b \leq n$ and $1\leq c<d \leq n$ be such that $(a,b)$ and $(c,d)$ are Tamari tree ascents of $T$ with $a<c$. Then $A^{Tam}_T(a,b)$, $A^{Tam}_T(c,d)$, and $F^{Tam}_T(a,c)$ are pairwise disjoint.
\end{lemma}

\begin{proof}
Assume seeking contradiction that $A^{Tam}_T(a,b) \cap A^{Tam}_T(c,d) \neq \emptyset$. Then by \cref{lem:taminversionsadded}, $b=d$. Then $a\in T^b_i$ and $c\in T^b_j$ with $i\neq j$ since $a$ and $c$ are distinct children of $d$. Thus, $T^a$ and $T^c$ are disjoint. However, the intersection being non-empty then contradicts \cref{lem:taminversionsadded}. 

If $F^{Tam}_T(a,c)\neq \emptyset$, then $b=c$ and $a\in T^c_0$ by \cref{def:terttaminvsadded}. Thus, $F^{Tam}_T(a,c)$ is disjoint from $A^{Tam}_T(a,b)$ since $b\neq d$. $F^{Tam}_T(a,c)$ is also disjoint from $A^{Tam}_T(c,d)$ by \cref{lem:taminversionsadded} because every $e\in T^a\setminus 0$ is in $T^c_0$ since $a\in T^c_0$.
\end{proof}

In the following lemma, we show the first of two descriptions of $Z\vee Q$ for $T\precdot_{Tam} Z,Q$. The second description of $Z\vee Q$ is \cref{lem:tamaddedinvsetcharact} below. Our proof of \cref{lem:tamjoinoftwoats} is nearly identical to the proof of \cref{lem:joinoftwoats} since the {\staml} is a sublattice of {\swo}.

\begin{lemma}\label{lem:tamjoinoftwoats}
Let $T$ be an {\stt} and let $1\leq a<b \leq n$ and $1\leq c<d \leq n$ be such that $(a,b)$ and $(c,d)$ are distinct Tamari tree ascents of $T$. Suppose $\trot{T}{Z}{(a,b)}$ and $\trot{T}{Q}{(c,d)}$, then $\inv{Z\vee Q} = \tcp{{\tcp{\inv{T}}{(b,a)}}}{(d,c)}$.
\end{lemma}

\begin{proof}
Since the $s$-Tamari lattice is a sublattice of {\swo}, $Z\vee Q$ is the same {\sdt} in the {\staml} as in {\swo}. Thus, this proof is the same as the proof of \cref{lem:joinoftwoats}, but with \cref{lem:inversionsadded} and \cref{lem:nocommoninvsadded} replaced by \cref{lem:taminversionsadded} and \cref{lem:nocommontaminvsadded}, respectively.
\end{proof}

In the next lemma, we begin with $s$-Tamari Tree $T$ with distinct Tamari tree ascents $(a,b)$ and $(c,d)$ with $a<c$. We show $(c,d)$ is always a Tamari tree ascent of the $s$-Tamari rotation of $T$ along $(a,b)$. We also show that the only way that $(a,b)$ ceases to be a Tamari tree ascent of the $s$-Tamari rotation of $T$ along $(c,d)$ is if $b=c$ and $a$ is the $0$th child of $c$ in $T$. In contrast with the four possibilities we say in \cref{lem:stopbeingascent} for {\swo}, there are only two possibilities in the {\staml}. These turn out to characterize which {\staml} intervals have Hasse diagrams that are diamonds and that are pentagons. The proof is simpler than that of \cref{lem:stopbeingascent} because Tamari tree ascents are pairs of a parent and child.

\begin{lemma}\label{lem:tamstopbeingascent}
Let $T$ be a {\stt}. Let $1\leq a<b \leq n$ and $1\leq c<d \leq n$ be such that $(a,b)$ and $(c,d)$ are Tamari tree ascents of $T$ with $a<c$. Let $\trot{T}{Z}{(a,b)}$ and $\trot{T}{Q}{(c,d)}$. If $(a,b)$ is not a Tamari tree ascent of $Q$, then $b=c$ and $a$ is the $0$th child of $c$. Moreover, $(c,d)$ is a Tamari tree ascent of $Z$.
\end{lemma}

\begin{proof}
By \cref{rmk:stamaritreepic}, the $s$-Tamari rotation along $(a,b)$ changes nothing above $c$ in $T$. Thus, $c$ is still a non-right most child of $d$ in $Z$ so $(c,d)$ is a Tamari tree ascent of $Z$. Because $a<c$, there are only two ways that $(a,b)$ might not be a Tamari tree ascent of $Q$. Either (1) $a\in Q^b_{s(b)}$ or (2) $a$ is not a child of $b$ in $Q$. 

For (1), we note that $a\in T^b_{j}$ for some $j<s(b)$ since $(a,b)$ is a Tamari tree ascent of $T$. Then by \cref{lem:taminversionsadded}, $a\in Q^b_{s(b)}$ implies $b=d$ and $a\in T^c$. Then, however, since $a<c$, $(a,d)$ being a Tamari tree ascent of $T$ contradicts \cref{lem:nolowertamascents}. Thus, (1) cannot occur. For (2), \cref{rmk:stamaritreepic} implies $a$ is a child of $b$ in $T$, but not a child of $b$ in $Q$ if and only if $b=c$ and $a$ is the $0$th child of $b$ in $T$. This is precisely the conclusion of this lemma.
\end{proof}

Next we give a second description of $Z\vee Q$ for $T\precdot_{Tam}Z.Q$, this time in terms of explicit multi-inversion sets instead of the transitive closure. The first description of $Z\vee Q$ we \cref{lem:tamjoinoftwoats} above. We use the same main idea as in the proof of \cref{lem:addedinvsetcharact} for {\swo} and construct the chain of length three that occurs in the intervals $[T,Z\vee Q]_{Tam}$ which have Hasse diagrams that are pentagons. 

\begin{lemma}\label{lem:tamaddedinvsetcharact}
Let $T$ be an {\stt}. Let $1\leq a<b \leq n$ and $1\leq c<d \leq n$ be such that $(a,b)$ and $(c,d)$ are Tamari tree ascents of $T$ with $a<c$. Suppose $\trot{T}{Z}{(a,b)}$ and $\trot{T}{Q}{(c,d)}$. Then $\inv{Z\vee Q} - \inv{T} = A^{Tam}_T(a,b)\cup A^{Tam}_T(c,d)\cup F_T^{Tam}(a,c) $.
\end{lemma}

\begin{proof}
If $(a,b)$ is a Tamari tree ascent of $Q$, then a similar argument to that in the proof of \cref{lem:addedinvsetcharact}, but with the corresponding lemmas for $s$-Tamari trees shows that the result holds.

If $(a,b)$ is not a Tamari tree ascent of $Q$, then $b=c$ and $a$ is the $0th$ child of $c$ by \cref{lem:tamstopbeingascent}. A similar argument to that in the proof of \cref{lem:addedinvsetcharact} using transitivity shows that $A^{Tam}_T(a,b)\cup A^{Tam}_T(c,d)\cup F_T^{Tam}(a,c) \subseteq \inv{Z\vee Q} -\inv{T}$. Thus, it suffices two show there is an {\stt} $P'$ with $\inv{P'} -\inv{T} = A^{Tam}_T(a,b)\cup A^{Tam}_T(c,d)\cup F_T^{Tam}(a,c)$. We claim there is a saturated chain \[\trot{T}{Q}{(c,d)}\trot{}{P}{(a,d)} \trot{}{P'}{(a,c)}. \] 

Since $a$ is the $0th$ child of $c$, $a$ is the $0th$ child of $d$ in $Q$ by \cref{rmk:stamaritreepic}. Thus, $(a,d)$ is a tree ascent of $Q$. Then, again by \cref{rmk:stamaritreepic}, $a$ is the $0th$ child of $c$ in $P$. Hence, $(a,c)$ is a Tamari tree ascent of $P$. Thus, we have the claimed saturated chain. Now we apply \cref{lem:taminversionsadded} at each step of the chain which gives $\inv{P'} -\inv{T} = A^{Tam}_T(c,d)\cup A^{Tam}_Q(a,d)\cup A_P^{Tam}(a,c)$. Now by \cref{rmk:stamaritreepic} we have $A^{Tam}_Q(a,d)= F_T(a,c)$ and $A_P^{Tam}(a,c)=A_T^{Tam}(a,c)$. Thus, $\inv{P'} -\inv{T} = A^{Tam}_T(c,d)\cup A^{Tam}_T(a,c)\cup F_T^{Tam}(a,c)$ and these sets are pairwise disjoint by \cref{lem:nocommontaminvsadded}.
\end{proof}

In the next lemma, we show that the only atoms in $[T,Z\vee Q]_{Tam}$ with $T\precdot_{Tam}Z,Q$ are $Z$ and $Q$ using \cref{lem:tamaddedinvsetcharact}.

\begin{lemma}\label{lem:tamnootheratoms}
Let $T$ be an {\sdt}. Let $1\leq a<b \leq n$ and $1\leq c<d \leq n$ be such that $(a,b)$ and $(c,d)$ are Tamari tree ascents of $T$ with $a<c$. Suppose $\trot{T}{Z}{(a,b)}$ and $\trot{T}{Q}{(c,d)}$, then $Z$ and $Q$ are the only atoms in $[T,Z\vee Q]_{Tam}$.
\end{lemma}

\begin{proof}
Assume $T'\in [T,Z\vee Q]_{Tam}$ and $T\precdot_{Tam} T'$ with $T'\neq Z,Q$. Let $(e,f)$ be the Tamari tree ascent of $T$ corresponding to $T'$. By \cref{lem:tamaddedinvsetcharact}, $(f,e)\in A^{Tam}_T(a,b)\cup A^{Tam}_T(c,d)\cup F_T^{Tam}(a,c)$. $(e,f)\neq (a,b),(c,d)$ since $T'\neq Z,Q$. Any other pair $(f,e)\in A^{Tam}_T(a,b)\cup A^{Tam}_T(c,d)\cup F_T^{Tam}(a,c)$ being a Tamari tree ascent of $T$ contradicts \cref{lem:nolowertamascents} because either $f=b$ or $f=d$ and $e$ is below $a$ or $c$ in $T$ and so cannot be a child of $f$.
\end{proof}

In the subsequent two lemmas, we show the {\staml} intervals of the form $[T,Z\vee Q]_{Tam}$ where $T\precdot_{Tam}Z,Q$ have Hasse diagrams that are either diamonds or pentagons and that the labeling of \cref{def:sblabelingtamari} satisfies the definition of SB-labeling. These two lemmas combine to prove our labeling is an SB-labeling of the {\staml}.

\begin{lemma}\label{lem:tamdiamondints}
Let $T\precdot_{Tam} Z,Q$ be cover relations in the {\staml} corresponding to $\trot{T}{Z}{(a,b)}$ and $\trot{T}{Q}{(c,d)}$ for distinct Tamari tree ascents of $(a,b)$ and $(c,d)$ of $T$. Suppose $(a,b)$ is a Tamari tree ascent of $Q$. Then $[T, Z\vee Q]_{Tam}$ has Hasse diagram which is a diamond and the edge labeling of \cref{def:sblabelingtamari} on its two maximal chains satisfies \cref{def:sblabeling}.
\end{lemma}

\begin{proof}
Similarly to the corresponding proof in {\swo}, we use \cref{lem:tamjoinoftwoats} to show $\trot{Z}{Z\vee Q}{(c,d)}$ and $\trot{Q}{Z\vee Q}{(a,b)}$. Hence, $Z,Q\precdot_{Tam} Z\vee Q$ Thus, $T\precdot_{Tam} Z \precdot_{Tam} Z\vee Q$ and $T \precdot_{Tam} Q \precdot_{Tam} R\vee Q$ are two distinct saturated chains from $T$ to $Z\vee Q$. To show there is not a third such saturated chain it suffices to show there is not a third atom in the interval $[T,Z\vee Q]_{Tam}$, but this is \cref{lem:tamnootheratoms}. Hence, the above chains are the only two saturated chains from $T$ to $Z\vee Q$.
        
Now we only need observe that the label sequences of the saturated chains $T\precdot_{Tam} Z \precdot_{Tam} Z\vee Q$ and $T \precdot_{Tam} Q \precdot_{Tam} Z\vee Q$ are $a,c$ and $c,a$, respectively. Therefore, \cref{def:sblabeling} is satisfied.
\end{proof}

\begin{lemma}\label{lem:tampentaints}
Let $T\precdot_{Tam} Z,Q$ be cover relations in the {\staml} corresponding to $\trot{T}{Z}{(a,b)}$ and $\trot{T}{Q}{(c,d)}$ for Tamari tree ascents $(a,b)$ and $(c,d)$ of $T$ with $a<c$. Suppose $(a,b)$ is not a Tamari tree ascent of $Q$. Then $[T, Z\vee Q]_{Tam}$ has Hasse diagram which is a pentagon and the edge labeling of \cref{def:sblabelingtamari} on its two maximal chains satisfies \cref{def:sblabeling}.
\end{lemma}

\begin{proof}
By \cref{lem:tamstopbeingascent}, $b=c$ and $a$ is the $0$th child of $c$. Again by \cref{lem:tamjoinoftwoats}, we have the saturated chain $T\precdot_{Tam} Z \precdot_{Tam} Z\vee Q$ given by the $s$-Tamari rotations $\trot{T}{Z}{(a,b)}$ and $\trot{Z}{Z\vee Q}{(c,d)}$. By the proof of \cref{lem:tamaddedinvsetcharact}, we have a saturated chain \[\trot{T}{Q}{(c,d)}\trot{}{P}{(a,d)} \trot{}{Z\vee Q}{(a,c)}. \] We note that by \cref{lem:taminversionsadded} $Z\preceq_{Tam} P$. Thus, to show the Hasse diagram of $[T, Z\vee Q]_{Tam}$ is a pentagon, it suffices to show there are no other elements in the interval besides $T,Z,Q,P,Z\vee Q$. To show there are no other elements in the interval, it suffices to show there are no other atoms in $[T,Z\vee Q]_{Tam}$ besides $Z$ and $Q$ and that there are no other atoms in $[Q,Z\vee Q]_{Tam}$ besides $P$. The fact that there are no atoms of $[T,Z\vee Q]_{Tam}$ besides $Z$ and $Q$ is \cref{lem:tamnootheratoms}. Similarly to the proof of \cref{lem:pentaints0} for {\swo}, \cref{lem:tamaddedinvsetcharact} implies the existence of an atom in $[Q,Z\vee Q]_{Tam}$ besides $P$ would contradict \cref{lem:nolowertamascents}. Hence, the Hasse diagram  of the interval is a pentagon whose only maximal chains are the two already shown. 

The label sequences for the maximal chains $T\precdot_{Tam} Z \precdot_{Tam} R\vee Q$ and $T\precdot_{Tam} Q \precdot_{Tam} P \precdot_{Tam} Z\vee Q$ are $a,c$ and $c,a,a$, respectively. These label sequences satisfy \cref{def:sblabeling}. 
\end{proof}

The previous two lemmas together prove the labeling of \cref{def:sblabelingtamari} is an SB-labeling.

\begin{theorem}\label{thm:tamsblabelthm}
Let $T\precdot_{Tam} Z$ be a cover relation in the {\staml}. Let $\trot{T}{Z}{(a,b)}$ be the $s$-Tamari rotation of $T$ along the Tamari tree ascent $(a,b)$ of $T$ associated to $T\precdot_{Tam} Z$ by \cref{thm:stamaricovers}. Let $\lambda$ be the edge labeling $\lambda(T, Z)= a$. Then $\lambda$ is an SB-labeling of the {\staml}.
\end{theorem}

\begin{proof}
Condition (i) of \cref{def:sblabeling} is satisfied by \cref{rmk:nosamebottomascents}. \cref{lem:tamstopbeingascent}, \cref{lem:tamdiamondints}, and \cref{lem:tampentaints} together imply conditions (ii) and (iii) of \cref{def:sblabeling} are satisfied proving the theorem.
\end{proof}

\cref{thm:tamsblabelthm} and \cref{thm:sbthm} prove a characterization of the homotopy type of open intervals in the {\staml} and so also characterize its M{\"o}bius function.

\begin{corollary}\label{cor:tamhomotopandmobius}
Let $T\preceq_{Tam} Z$ in the {\staml}. Then $\Delta(T,Z)_{Tam}$, the order complex of the open interval $(T,Z)_{Tam}$, is homotopy equivalent to a ball or a sphere of some dimension. Moreover, the M{\"o}bius function of the {\staml} satisfies $\mu_{Tam}(T,Z) \in \set{-1,0,1}$.
\end{corollary}

Furthermore, we give the analogous intrinsic description of open $s$-Tamari intervals whose order complexes are homotopy spheres as for $s$-weak order. We begin with a lemma characterizing the join of atoms in a closed interval $[T,Z]_{Tam}$.

\begin{lemma}\label{lem:tamlesssharpcharacspheres}
If $T \prec_{Tam} Z$, then $Z$ is the join of the atoms in $[T,Z]_{Tam}$ if and only if \[ \inv{Z} = \left( \inv{T} + A^{Tam}_T(a_1,b_1) + \dots + A^{Tam}_T(a_l,b_l) \right)^{tc} \] where $(a_1,b_1), \dots, (a_l,b_l)$ are the Tamari tree ascents of $T$ such that $(b_i,a_i)\in \inv{Z}-\inv{T}$. Moreover, the number of atoms in the interval $[T,Z]_{Tam}$ is $l$ regardless of whether or not $Z$ is the join of atoms in the interval.
\end{lemma}

\begin{proof}
The number of atoms follows from \cref{thm:stamaricovers}, the characterization of cover relations in the {\staml}. The rest of the statement follows from the same argument as in the proof of \cref{lem:lesssharpcharacspheres} with the lemmas about {\swo} replaced by the corresponding lemmas for the {\staml} because the {\staml} is a sublattice of {\swo}.
\end{proof}

We conclude with the theorem characterizing the open $s$-Tamari intervals which are homotopy equivalent to spheres.

\begin{theorem}\label{thm:tamcountthespheres}
If $T \prec Z$, then $\Delta (T,Z)_{Tam}$ is homotopy equivalent to a sphere if and only if \[\inv{Z} = \left(\inv{T} + A^{Tam}_T(a_1,b_1) + \dots + A^{Tam}_T(a_l,b_l) \right)^{tc}\] where $(a_1,b_1), \dots, (a_l,b_l)$ are the Tamari tree ascents of $T$ such that $(b_i,a_i)\in \inv{Z}-\inv{T}$. Moreover, in this case the dimension of the sphere is $l-2$.
\end{theorem}

\begin{proof}
This follows from combining \cref{lem:tamlesssharpcharacspheres} and \cref{thm:sbthm}.
\end{proof}

\end{section}

\section*{Acknowledgements}
The author is grateful to Patricia Hersh for guidance and many helpful discussions throughout the course of this work. The author also thanks Joseph Doolittle for helpful discussions about the classical Tamari lattice at FPSAC 2019 where the author first encountered the work of Ceballos and Pons.

\bibliography{bib}

\vspace{15mm}

Stephen Lacina

Department of Mathematics, University of Oregon, Eugene, OR 97403

\textit{Email address:} slacina@uoregon.edu

\end{document}